\RequirePackage{fix-cm}
\documentclass[smallextended]{svjour3}       
\smartqed  
\usepackage{graphicx}
\usepackage{latexsym,amsmath,amssymb}
\usepackage{algorithm,algpseudocode,color}
\usepackage{booktabs,multirow,makecell}
\usepackage{hyperref}

\begin{document}

\title{Difference of convex algorithms for bilevel programs with applications in hyperparameter selection\thanks{This paper is dedicated to the memory of Olvi L. Mangasarian. } \thanks{The research of the first author was partially supported by NSERC. The second author was supported by a General Research Fund from Hong Kong Research Grants Council. The third author was  supported by the Pacific Institute for the
Mathematical Sciences (PIMS). The last author was supported by  NSFC (No. 12222106), Shenzhen Science and Technology Program (No. RCYX20200714114700072) and the
Guangdong Basic and Applied Basic Research Foundation
(No. 2022B1515020082) .}
}

\titlerunning{Difference of convex algorithms for bilevel programs}        

\author{Jane J. Ye        \and
        Xiaoming Yuan \and Shangzhi Zeng \and Jin Zhang
}


\institute{Jane J. Ye \at
              Department of Mathematics and Statistics, University of Victoria, Canada. \\
              \email{janeye@uvic.ca}           
           \and
           Xiaoming Yuan \at
           Department of Mathematics,  The University of Hong Kong, Hong Kong SAR, China.\\
           \email{xmyuan@hku.hk}
           \and
           Shangzhi Zeng \at
           Department of Mathematics and Statistics, University of Victoria, Canada.\\
           \email{zengshangzhi@uvic.ca}
           \and
           Corresponding author. Jin Zhang \at
           Department of Mathematics, SUSTech International Center for Mathematics,
           Southern University of Science and Technology, National Center for Applied Mathematics Shenzhen, Peng Cheng Laboratory, Shenzhen, China.\\
           \email{zhangj9@sustech.edu.cn}
}

\date{Received: date / Accepted: date}

\maketitle

\begin{abstract}
In this paper, we present difference of convex algorithms for solving bilevel programs in which the upper level objective functions are difference of convex functions, and the lower level programs are fully convex. 
This nontrivial class of bilevel programs provides a powerful modelling framework for dealing with  applications arising from hyperparameter selection in machine learning. Thanks to the full convexity of the lower level program,  the value function of the lower level program turns out to be convex and hence the bilevel program can be reformulated as a difference of convex bilevel program. We propose two algorithms for solving the reformulated difference of convex program and show their convergence to stationary points under very mild assumptions. Finally we conduct numerical experiments to a bilevel model of support vector machine classification. 

\end{abstract}

\section{Introduction}	

Bilevel programs are a class of hierarchical optimization problems which have constraints containing a lower-level optimization problem  parameterized by upper-level variables. Bilevel programs capture a wide range of important applications in various fields including Stackelberg games and moral hazard problems in economics (\cite{Mirrlees,Stackelberg}),  hyperparameter selection and meta learning  in machine learning (\cite{Franceschi,Kunapuli,kunapuli2008classification,kunapuli2008bilevel,Liu,Liu2,Mooreth,Moore,HO2020}).  More applications can be found in the monographs \cite{bard1998practical,Dempe2002,dempe2015bilevel,Shimizu},  the survey on bilevel optimization \cite{ColsonMarcotteSavard,Dempebook} and the references within.

In this paper, we  develop some numerical algorithms for solving the following  difference of convex (DC) bilevel program:
\begin{equation*}\label{general_problem}
	{\rm (DCBP)}~~~~~~~~~~\begin{aligned}
		\min_{x \in \mathbb{R}^n, y \in \mathbb{R}^m}  ~~& F(x,y):=F_1(x,y)-F_2(x,y) \\
		s.t. ~~~~& x\in X,  y \in S(x),
	\end{aligned}
\end{equation*}
with $S(x)$ being the set of optimal solutions of the lower level problem,
\[
\begin{aligned}
	(P_x): \quad \min_{y \in Y}~&f(x,y) \\s.t.~& g(x,y) \le 0,
\end{aligned}
\]
where   $X \subseteq \mathbb{R}^n$ and $Y \subseteq \mathbb{R}^m$ are nonempty closed sets, $g:=(g_1,\dots, g_l)$,  all  functions $g_i : \mathbb{R}^n \times \mathbb{R}^m \rightarrow \mathbb{R}, ~ i = 1,\ldots,l$ are  convex  on  an open convex set containing the set $X\times Y$, and 
the functions $F_1,F_2, f : \mathbb{R}^n \times \mathbb{R}^m \rightarrow \mathbb{R}$ are convex  on an open convex set containing the set
$$C:=\{(x,y)\in X\times Y: g(x,y) \le 0\}.$$ To ensure the  bilevel program is well-defined,  we assume that $S(x)\not =\emptyset$ for all $x\in X$.
{Moreover we assume that   for all  $x$ in  an open convex set   ${\cal O}\supseteq X$,  the feasible region for the lower level program
	${\cal F}(x):=\{y\in Y: g(x,y)\leq 0\}$ is nonempty and  the lower level objective function $f(x,y)$ is bounded below on ${\cal F}(x)$.} 
	
Although the objective function  in the DC bilevel program we consider must be a DC function,  this setting is general enough to capture many cases of practical interests. 
In particular
any  lower $C^2$ function (i.e., a function which can be  locally written as a supremum of a family of $C^2$ functions) and $C^{1+}$ function  (i.e., a differentiable function whose gradient is locally Lipschitz continuous) are DC functions and the class of DC functions is closed under many operations encountered frequently in optimization; see, e.g., \cite{DC_overview,ThiDinh}. In the lower level program, we assume all functions are fully convex, i.e., convex in both variables $x$ and $y$. However as pointed out by \cite[Example 1 and Section 5]{LLNashBilevel},  using some suitable reformulations one may turn a non-fully convex lower level program into a fully convex one. Also as demonstrated in this paper, the bilevel model for hyperparameter selection problem can be reformulated as a bilevel program where the lower level is fully convex.

Solving bilevel programs numerically is extremely challenging. It is known that even when all defining functions are linear, the computational complexity is already NP-hard \cite{BenAyed}.
If all defining functions are smooth and the lower level program is convex with respect to the lower level variable, the first order approach was popularly used to replace the lower level problem by its first order optimality condition  and to solve the resulting problem as the mathematical program with equilibrium constraints (MPEC);  see e.g. \cite{allende2013solving,bard1998practical,Dempebook,MPEC1,MPEC2}. The first order approach may be problematic since  it may not provide an equivalent reformulation to the original bilevel program if  only local (not global) optimal solutions are considered; see \cite{Dem-Dut}. Moreover even in the case of a fully convex lower level program, \cite[Example 1]{LLNashBilevel} shows that it is still possible that a local optimal solution of  the corresponding MPEC does not correspond to a local optimal solution of the original bilevel program. Recently some numerical algorithms have been introduced for solving bilevel programs where the lower level problem is not necessarily convex in the lower level variable; see e.g., \cite{LinXuYe,nie2017bilevel,nie2021bilevel}. However these approaches have limitations in the numbers of variables in the bilevel program.
In most of  literature on numerical algorithms for solving bilevel programs, smoothness of all defining functions are assumed.
In some special cases, non-smoothness can be dealt with by introducing auxiliary variables and constraints to reformulate a nonsmooth lower level program as a smooth constrained lower level program. But using such an approach  the numbers of variables or constraints would increase. 


Our research on the DC bilevel program is motivated by a number of important applications in model selection and hyperparameter learning. Recently in the statistical learning, the regularization parameters has been successfully used, e.g., in the  least absolute shrinkage and selection operator (lasso) method for regression and support vector machines (SVMs) for classification. However the regularization parameters have to be set {\it a priori} and the choice of these parameters dramatically affects results on the model selection. The most commonly used method for selecting these parameters is the so-called $T$-fold cross validation.
By $T$-fold cross validation, a data set $\Omega$ is {randomly} partitioned into $T$ pairwise disjoint subsets called the validation sets  $\Omega_{val}^t$, $t=1,\dots, T$.
For  each fold $t=1,\dots, T$, a subset of $\Omega$ denoted by  $\Omega_{trn}^t:= \Omega\backslash\Omega_{val}^t$ is used for training and  the validation set  $\Omega_{val}^t$ is used for testing the result. 
Take the SVM problem for example,  suppose the data set $\Omega=\{ (\mathbf{a}_j,b_j)\}_{j=1}^\ell  $ where $\mathbf{a}_j \in \mathbb{R}^{n},$ and the labels $b_j = \pm 1$ indicate the class membership.  For each hyperparameters $\lambda>0, \bar{\mathbf{w}}$ and each fold $t=1,\dots, T$, the following SVM problem can be solved. 
$$(P_{\lambda,  \bar{\mathbf{w}}}^t)~~~~~\underset{\tiny \begin{matrix}
					-\bar{\mathbf{w}} \le \mathbf{w} \le \bar{\mathbf{w}} \\ c \in \mathbb{R}
						\end{matrix} }{\mathrm{min}}  \left\{ \frac{\lambda}{2}\|\mathbf{w}\|^2 + \sum_{j\in \Omega_{trn}^t}\max( 1 -b_j( \mathbf{a}_j^T\mathbf{w}-c),0) \right \}.$$
						The desirable hyperparameters  $ \lambda^*$ and  $ \bar{\mathbf{w}}^*$ can be selected by minimizing some measure of validation accuracy over all folds
						such as
						$$\Theta( \mathbf{w}^1, \dots,\mathbf{w}^T ,\mathbf{c}) := \frac{1}{T} \sum_{t=1}^{T} \frac{1}{|\Omega_{val}^t|} \sum_{j\in \Omega_{val}^t}\max(1-b_j(\mathbf{a}_j^T\mathbf{w}^t_{\lambda,  \bar{\mathbf{w}}}-c^t_{\lambda,  \bar{\mathbf{w}}}),0),  $$
						where $|M|$ denotes the number of elements in set $M$ and $(\mathbf{w}^t_{\lambda,  \bar{\mathbf{w}}},  c^t_{\lambda,  \bar{\mathbf{w}}})$ denotes a solution to the SVM problem $(P_{\lambda,  \bar{\mathbf{w}}}^t)$.  Here the cross validation error is based on the hinge loss function.  Other possible functions that  can be used for cross validation error  can be found in \cite{Bennett,Kunapuli,kunapuli2008classification}.
In fact,  the hyperparameter selection for SVM has been proposed as the following  bilevel program with $T$ lower level programs by \cite{Kunapuli,kunapuli2008classification}:
\begin{equation*}\label{SVBP}
	\begin{aligned}
		&\min
		_{\lambda,\bar{\mathbf{w}},\mathbf{w}^1,\dots, \mathbf{w}^T, \mathbf{c}} 
		~~ 
		\Theta( \mathbf{w}^1, \dots,\mathbf{w}^T ,\mathbf{c})\\
		&~~ \qquad \begin{aligned}
			 s.t. ~~ &\lambda_{lb} \le \lambda \le \lambda_{ub}, \quad \bar{\mathbf{w}}_{lb} \le \bar{\mathbf{w}} \le \bar{\mathbf{w}}_{ub},   \\
			& \mathrm{and~for}~t = 1,\ldots,T:\\
			& (\mathbf{w}^t, c^t) \in \underset{\tiny \begin{matrix}
					-\bar{\mathbf{w}} \le \mathbf{w} \le \bar{\mathbf{w}} \\ c \in \mathbb{R}	\end{matrix} }{\mathrm{argmin}} \left\{ \frac{\lambda}{2}\|\mathbf{w}\|^2 + \sum_{j\in \Omega_{trn}^t}\max( 1 -b_j( \mathbf{a}_j^T\mathbf{w}-c),0) \right\},
		\end{aligned}
	\end{aligned}
\end{equation*}
where
$\mathbf{c} \in \mathbb{R}^T$ is the vector with
$c^t$ as the $t$th component.
Here $\lambda_{lb}, \lambda_{ub} $  are given  positive numbers and $ \bar{\mathbf{w}}_{lb}, \bar{\mathbf{w}}_{ub}$ are given  vectors in $\mathbb{R}^n$. It is easy to see that by changing the variable $\lambda$ to  $\mu := \frac{1}{\lambda}$ 
we can reformulate the above SV bilevel model selection  equivalently as the following bilevel program with a single lower level program
\begin{equation*}
{\rm (SVBP)}~~~~~~~~	\begin{aligned}
		&\min
		_{\mu,\bar{\mathbf{w}},\mathbf{w}^1,\dots, \mathbf{w}^T, \mathbf{c}} 
		~~ 
		\Theta( \mathbf{w}^1, \dots,\mathbf{w}^T ,\mathbf{c})\\
		&~~ \qquad \begin{aligned}
			 s.t. ~~ &\frac{1}{\lambda_{ub}} \le \mu \le \frac{1}{\lambda_{lb}}, \quad \bar{\mathbf{w}}_{lb} \le \bar{\mathbf{w}} \le \bar{\mathbf{w}}_{ub},   \\
			& (\mathbf{w}^1,\dots,  \mathbf{w}^T,\mathbf{c}) \in S(\mu ,\bar{\mathbf{w}}),  
		\end{aligned}
	\end{aligned}
\end{equation*}
where $S(\mu, \bar{\mathbf{w}})$ is the set of optimal solutions of the lower level problem
$$(P_{\mu, \bar{\mathbf{w}}})~~~~~ \underset{\tiny \begin{matrix}
					-\bar{\mathbf{w}} \le \mathbf{w}^t \le \bar{\mathbf{w}} \\ c^t \in \mathbb{R}\\
					t=1,\dots, T	\end{matrix} }{\mathrm{min}} \left\{ \sum_{t=1}^T \left (\frac{\|\mathbf{w}^t\|^2}{2\mu} + \sum_{j\in \Omega_{trn}^t} \max( 1 -b_j( \mathbf{a}_j^T\mathbf{w}^t-c^t),0) \right )\right\}.$$
Moreover using the fact that a function in the form $\phi(\mathbf{x}, \mu)= {\|\mathbf{x}\|^2}/{\mu}$ with $\mu>0$ is convex as a perspective function \cite[Example 3.18]{cvxbook}, the above bilevel program has a fully convex lower level program and all required assumptions hold; see the details in Section \ref{sec_SV}. 
The classical $T$-fold cross validation method for selecting hyperparameters usually implements a {\it grid search}: training $T$ models at each point of a discretized parameter space in order to find an approximate optimal parameter. This method has many drawbacks and limitations. In particular its computational complexity scales exponentially with the number of hyperparameters and the number of grid points for each hyperparameter. Hence the grid search method is not practical for problem requiring several hyperparameters, including our SV  bilevel model selection where the hyperparameters $(\mu,\bar{\mathbf{w}})$ are $n + 1$ dimentional.  To deal with limitations of grid search, introduced first in \cite{Bennett} in 2006,  the bilevel program has been used to model hyperparameter selection problems in \cite{Bennett,Kunapuli,kunapuli2008classification,kunapuli2008bilevel,Mooreth,Moore}.

The above fully convex transformation using the perspective function can be applied to some other model hyperparameter selection problems, for example, the $T$-fold cross validation Lasso problem.

This paper is motivated by an interesting fact that, under our problem setting, the value function of the lower level in (DCBP) defined  by
\begin{equation*}\label{optimal_value_function}
	v(x) := \inf_{y \in Y} \left\{ f(x,y) ~ s.t.~ g(x,y) \le 0  \right\}
\end{equation*}
is convex and locally Lipschitz continuous on $X$.
We take full advantage of this convexity and use 
the value function approach first proposed in \cite{Outrata1990} for a numerical purpose and further used to study optimality conditions in \cite{ye1995} to
reformulate (DCBP)  as the following DC program:
\begin{equation*}\label{general_reformulated_problem}
	\begin{aligned}
		{\rm (VP)}  \quad\qquad \min_{(x,y)\in C} ~~& F_1(x,y)-F_2(x,y)\\
		s.t. ~~~~& f(x,y) - v(x) \le 0.
	\end{aligned}
\end{equation*}
Unfortunately, due to the value function constraint,  (VP) violates the usual constraint qualification such as the nonsmooth Mangasarian Fromovitz constraint qualification (MFCQ) at each feasible point, see \cite[Proposition 3.2]{ye1995} for the smooth case and Proposition \ref{Prop4.3} for the nonsmooth case. It is well-known that  convergence of the difference of convex algorithm (DCA) is only guaranteed under 
constraint qualifications such as the extended MFCQ,  which is MFCQ extended to infeasible points; see, e.g., \cite{le2014dc}. To deal with this issue, we consider the following approximate bilevel program
\begin{equation*}\label{general_reformulated_problem_epsilon}
	\begin{aligned}
		({\rm VP})_{\epsilon} \quad\qquad \min_{(x,y)\in C} ~~& F_1(x,y)-F_2(x,y)\\
		s.t. ~~~~& f(x,y) - v(x) \le \epsilon,
	\end{aligned}
\end{equation*}
for some $\epsilon>0$. Such a relaxation strategy has been used for example in \cite{LinXuYe} based with the reasoning  that in numerical algorithms one usually obtain an inexact optimal solution anyway and the solutions of $({\rm VP})_{\epsilon}$ approximate a solution of the original bilevel program (VP) as $\epsilon$ approaches zero. In this paper we will show that EMFCQ  holds for problem $({\rm VP})_{\epsilon}$ when $\epsilon>0$ automatically. Hence we propose to solve problem $({\rm VP})_{\epsilon}$ with $\epsilon\geq 0$. When $\epsilon>0$, the convergence of our algorithm to stationary points is guaranteed and when  $\epsilon=0$, the convergence is not 
guaranteed but it could still converge if the penalty parameter sequence is bounded.

Using DCA approach, at each iterate point $(x^k,y^k)$,  one  linearises the concave part of the function, i.e., the functions $F_2(x,y), v(x)$ by using an element of the subdifferentials $\partial F_2(x^k,y^k), \partial v(x^k)$ and solve a resulting convex subproblem. 
The value function is an implicit function. How do we obtain an element of the subdifferential  $\partial v(x^k)$? At current iterate $x^k$, assuming we can solve  the lower level problem $(P_{x^k})$ with a global minimizer $\tilde{y}^k$ and  a corresponding Karush-Kuhn-Tucker (KKT) multiplier denoted by $\gamma^k$. Suppose that 
the following partial derivative formula holds:
\begin{eqnarray}
	\partial f( x, y)= \partial_x f( x, y)\times \partial_y f( x, y) ,&& 
	\partial g_i( x,y)= \partial_x g_i( x, y)\times \partial_y g_i( x,y) \label{partialdnew1}
\end{eqnarray}
at $(x,y)=(x^k,\tilde{y}^k)$.
Then since by convex analysis $$ \partial_x f(x^k,\tilde{y}^k) + \sum_{i = 1}^{l} \gamma^k_i \partial_x g_i(x^k,\tilde{y}^k)\subseteq \partial v(x^k),$$ we can select an element of  $\partial v(x^k)$ from the set $$\partial_x f(x^k,\tilde{y}^k) + \sum_{i = 1}^{l} \gamma^k_i \partial_x g_i(x^k,\tilde{y}^k)$$ and use it to linearize the value function. We then solve the resulting convex subproblem approximately to obtain a new iterate $(x^{k+1},y^{k+1})$. 
Thanks to recent developments in large-scale convex programming, using this approach we can deal with  a large scale DC bilevel program.  

Now we summarize  our contributions as follows.
\begin{itemize}

	\item We propose two new algorithms for solving DC program. These algorithms have modified the classical DCA in two ways. First, we add a proximal term in each convex subproblem so the the objective function is strongly convex and at each iterate point, only an approximate solution for the convex subproblem is solved. Second, our penalty parameter update is simplier.
	\item We have laid down all theoretical foundations from convex analysis that are required for our algorithms to work. In particular we have demonstrated that  under the minimal assumptions that we specify   for problem (DCBP), the value function is convex and locally Lipschitz on set $X$ automatically.
	\item Using the two new algorithms for solving DC program, we propose two corresponding algorithms to solve problem (DCBP). Our algorithms hold under very mild and natural assumptions. In particular we allow  all defining functions to be nonsmooth and we do not require any constraint qualification to hold for the lower level program.
	The main assumptions we need are only the partial derivative formula (\ref{partialdnew1}) which holds under many practical situations (see Proposition \ref{partiald} for sufficient conditions) and the existence of a KKT multiplier for the lower level program under each iteration. 
	\item Taking advantage of large scale convex programming, our algorithm can handle high dimensional hyperparameter selection   problems. To test effectiveness of our algorithm, we have tested it in {the SV bilevel model selection (SVBP)}.
	Our results compare favourably with the MPEC approach \cite{Kunapuli,kunapuli2008classification,kunapuli2008bilevel}. 

\end{itemize}
This paper is organized as follows. In Section 2 
we propose   two modified DCAs and study their convergence to stationary points for a class of general DC programs.  In Section 3,  we derive  explicit conditions for the bilevel program under which the algorithms introduced in Section 3 can be applied.
Numerical experiments on the SV bilevel model selection is conducted on Section 4.  Section 5 concludes the paper.

\section{Modified inexact proximal DC algorithms}\label{sec:DC}
In order to solve the (relaxed) value function reformulation of problem DCBP, in this section we propose numerical algorithms to solve the following difference of convex program:
\begin{eqnarray*}
	({\rm DC})~~~~~~~\min_{ z\in \Sigma} && f_0(z):=g_0(z)-h_0(z)\\
	s.t. && f_1(z):=g_1(z)-h_1(z)\leq 0,
\end{eqnarray*}
where $\Sigma$ is a closed convex subset of $\mathbb{R}^d$ and $g_0(z),h_0(z),g_1(z),h_1(z):\Sigma \rightarrow \mathbb{R}$ are   convex functions.
Although the results in this section can be generalized to the case where there are more than one inequality in a straight-forwarded manner,  to simplify the notation and concentrate on the main idea we assume there is only one inequality constraint in problem (DC). Our algorithms are modifications of   the classical DCA (see \cite{le2014dc}).
Recently, \cite{pang2017} studied problem (DC) where $h_1$ is a maximum of finitely many smooth convex functions and proposed an algorithm for finding B-stationary points of it.

Before   introducing our algorithms and conduct the convergence analysis,  we recall some notations from convex analysis and variational analysis.
Let $\varphi(x):\mathbb{R}^n\rightarrow [-\infty,+\infty]$ be a convex function, and let $\bar x $ be a point where $\varphi$ be finite.  The subdifferential of $\varphi$ at $\bar x$ is a closed convex set defined by
$$\partial \varphi(\bar x):=\left \{ \xi \in \mathbb{R}^n| \ \varphi(x)\geq \varphi(\bar x)+\langle \xi, x-\bar x\rangle , \ \forall x \right \},$$
and a subgradient is an  element of the subdifferential.
For a function $\varphi:\mathbb{R}^n\times \mathbb{R}^m\rightarrow [-\infty,+\infty]$, we denote the partial subdifferential of $\varphi$ with respect to $x$ and $y$ by  $\partial_x \varphi(x, y) $ and $ \partial_y \varphi(x, y)$ respectively.
Let $\Sigma$ be a convex subset in $\mathbb{R}^n$ and $\bar x \in \Sigma$. The normal cone to $\Sigma$ at $\bar x$ is denoted by
$\mathcal{N}_\Sigma(\bar x)$.
Let $\delta_\Sigma(x)$ denote the indicator function of set $\Sigma$ at $x$.
The following partial subdifferentiation rule will be useful.
\begin{proposition}[Partial subdifferentiation]\label{partiald}
	Let $\varphi:\mathbb{R}^n\times \mathbb{R}^m\rightarrow [-\infty,+\infty]$ be a convex function and let $(\bar x,\bar y)$ be a point where $\varphi$ is finite. Then
	\begin{equation} \label{partialsubg} \partial \varphi(\bar x,\bar y) \subseteq \partial_x \varphi(\bar x,\bar y) \times \partial_y \varphi(\bar x,\bar y).\end{equation}
	The inclusion (\ref{partialsubg}) becomes an equality under one of the following conditions.
	\begin{itemize}
		\item[(a)] For every $\xi \in \partial_x \varphi (\bar x,\bar y)$, it holds that $ \varphi (x,y)-\varphi (\bar x, y)\geq \langle \xi, x-\bar x\rangle, \  \forall (x,y) \in \mathbb{R}^n\times \mathbb{R}^m .$
		\item[(b)] $\varphi(x,y) =\varphi_1(x) +\varphi_2(y)$.
		\item[(c)] For any $\varepsilon>0$, there is $\delta>0$ such that 
		\begin{eqnarray}
			\mbox{ either } 	&& \partial_x \varphi (\bar x,\bar y) \subseteq \partial_x \varphi (\bar x, y)  +\varepsilon B_{\mathbb{R}^n} \quad \forall y\in B(\bar y; \delta) \label{(7)} \\
			\mbox{ or } 	&& \partial_y \varphi (\bar x,\bar y) \subseteq \partial_y \varphi ( x, \bar y)  +\varepsilon B_{\mathbb{R}^m} \quad \forall x\in B(\bar x; \delta),\label{(8)}
		\end{eqnarray}
		where $B(\bar x; \delta)$ denotes the open ball centered at $\bar x$ with radius equal to $\delta$ and $B_{\mathbb{R}^n}$ denotes the open unit ball centered at the origin in $\mathbb{R}^n$.
		\item[(d)] $\varphi(x,y)$  is continuously differentiable  respect to one of the variables $x$ or   $y $  at    $(\bar x,\bar y)$.
	\end{itemize}
	Moreover 
	$(b)\Longrightarrow (a), (d) \Longrightarrow (c) \Longrightarrow(a).$
\end{proposition}
\begin{proof}The inclusion (\ref{partialsubg})  and its reverse under (a) follow directly from definitions of the convex subdifferential and the partial subdifferential.
	When $\varphi(x,y) =\varphi_1(x) + \varphi_2(y)$, we have that $\partial \varphi(x,y) =  \partial \varphi(x) \times \{0\} + \{0\} \times \partial \varphi(y) $. Hence obviously (b) implies (a). The implication of (d) to (c) is obvious.
	Now suppose that (\ref{(7)}) holds. Let $\xi \in \partial_x \varphi (\bar x,\bar y)$. Then according to (\ref{(7)}),  for any $\varepsilon>0$, there is $\delta>0$ such that $ \xi=\eta+\varepsilon e$, where $e \in B_{\mathbb{R}^n}$, and 
	$$\langle \xi, x-\bar x\rangle \leq \varphi (x,y)-\varphi(\bar x,y)+\varepsilon\|x-\bar x\|\quad \forall y\in B(\bar y; \delta) .$$
	Thanks to the convexity of $\varphi$, using the proof technique of \cite[Corollary 2.6 (c)]{ClarkeLSW}, we can easily show that (a) holds. The proof for the case where (\ref{(8)}) holds is similar and thus omitted. \qed
\end{proof}

Next, we first brief some solution quality characterizations for problem  (DC).
\begin{definition}\label{Defn3.1} Let $\bar z$ be a feasible solution of problem (DC). We say that $\bar z$ is a stationary/KKT point of problem (DC) if there exists a multiplier $\lambda\geq 0$ such that 
	\begin{eqnarray*} 
		&& 0\in \partial g_0(\bar z) -\partial h_0(\bar z) +\lambda( \partial g_1(\bar z) -\partial h_1(\bar z))+\mathcal{N}_\Sigma(\bar z),\\
		&& (g_1(\bar z)-h_1(\bar z))\lambda=0.
	\end{eqnarray*}
\end{definition}

\begin{definition}\label{Defn3.2} Let $\bar z$ be a feasible point of problem (DC). We say that the nonzero abnormal multiplier constraint qualification (NNAMCQ) holds at $\bar z$ for problem (DC) if either $f_1(\bar z) <0$ or $f_1(\bar z)= 0$ but 
	\begin{equation}\label{NNAMCQ}
		0 \not \in  \partial g_1(\bar{z}) -\partial h_1(\bar{z}) + \mathcal{N}_\Sigma(\bar{z}).
	\end{equation}
	Let $\bar z\in \Sigma$, we say that the extended no nonzero abnormal multiplier constraint qualification (ENNAMCQ) holds at $\bar z$ for problem (DC) if either $f_1(\bar z) <0$ or $f_1(\bar z)\geq 0$ but
	(\ref{NNAMCQ}) holds.
\end{definition}
Note that NNAMCQ (ENNAMCQ) is equivalent to  MFCQ  (EMFCQ) respectively; see e.g., \cite{Jourani}.

{ Denote by $\partial^c \varphi(x)$ the Clarke generalized gradient \cite{clarke1990optimization} of a locally Lipschitz function $\varphi$ at $x$.}   The  following optimality condition follows from the nonsmooth multiplier rule in terms of Clarke generalized gradients (see e.g. \cite{clarke1990optimization,Jourani}) and the fact that for two convex  functions $g,h$ which are Lipschitz around point $\bar z$, we have $\partial^c (g(\bar z)-h(\bar z))\subseteq \partial^c g(\bar z) -\partial^c h(\bar z)=\partial g(\bar z) -\partial h(\bar z)$.
\begin{proposition}\label{Thm3.1} Let $\bar z$ be a local solution of problem (DC). If NNAMCQ holds at $\bar z$ and all functions $g_0,g_1,h_0,h_1$ are Lipschitz around point $\bar z$, then $\bar z$ is a KKT point of problem (DC).
\end{proposition}

\subsection{Inexact proximal DCA with simplified penalty parameter update}

In this subsection we  propose an algorithm called inexact proximal DCA to solve  problem (DC) and show its convergence to stationary points.

By using the main idea of DCA which linearizes the concave part of the DC structure, we propose a sequential convex programming scheme as follows.
Given a current iterate $z^k\in \Sigma$ with  $k=0,1,\ldots$, we select a subdifferential $\xi^k_i\in \partial h_i(z^k)$, for $i=0,1$. 
Then we solve the following subproblem approximately and select $z^{k+1}$ as an approximate minimizer:
\begin{equation} \label{subp} 
	\begin{aligned}
	\min_{z \in \Sigma} ~ \tilde{\varphi}_k(z) :=\; & g_0(z) - h_0(z^k) -\langle \xi_0^k, {z-z^k}\rangle \\
	&+\beta_k \max\{g_1(z) -h_1(z^k)- \langle\xi_1^k, z-z^k\rangle, 0\} +\frac{\rho}{2} \|z-z^k\|^2,
	\end{aligned}
\end{equation}
where $\rho$ is a given positive constant and $\beta_k $ represents the adaptive penalty parameter.
Our scheme is similar to that of DCA2 in \cite{le2014dc} but   different in that  the subproblem (\ref{subp}) has a strongly convex objective function,  the subproblem is only solved approximately,  and  a simplier penalty parameter update is used. 
We propose the following two inexact conditions for choosing $z^{k+1}$ as an approximate solution to (\ref{subp}):
\begin{equation}\label{inexact1}
	\mathrm{dist}(0, \partial \tilde{\varphi}_k(z^{k+1}) + \mathcal{N}_\Sigma(z^{k+1})) \le \zeta_{k}, \quad \mbox{for some  } \zeta_k\geq 0  \mbox{ satisfying } \sum_{k = 0}^{\infty} \zeta_{k}^2 < \infty,
\end{equation}
and 
\begin{equation}\label{inexact2}
	\mathrm{dist}(0, \partial \tilde{\varphi}_k(z^{k+1}) + \mathcal{N}_\Sigma(z^{k+1})) \le \frac{\sqrt{2}}{2} \rho\|z^k - z^{k-1}\|,
\end{equation}
where $\mathrm{dist}(x, M)$ denotes the distance from a point $x$ to set $M$.

Using  above constructions, we are ready to propose the inexact proximal DCA (iP-DCA) in Algorithm \ref{ipDCA}. 
\begin{algorithm}[h]
	\caption{iP-DCA}\label{ipDCA}
	\begin{algorithmic}[1]
		\State Take an initial point $z^0\in \Sigma$; $\delta_\beta > 0$; an initial penalty parameter $\beta_0>0$, $tol>0$.
		\For{$k=0,1,\ldots$}{
			\begin{itemize}
				\item[1.] Compute $\xi^k_i\in \partial h_i(z^k)$, $i=0,1$.
				\item[2.] Obtain an inexact solution $z^{k+1}$ of (\ref{subp}) satisfying  \eqref{inexact1} or \eqref{inexact2}.
				\item[3.] Stopping test. Compute {$t^{k+1}:= \max\{g_1(z^{k+1}) -h_1(z^k)- \langle\xi_1^k, z^{k+1} - z^k\rangle, 0\}$}. Stop if $\max \{ \|z^{k+1}-z^k\|, t^{k+1}\} <tol$.
				\item[4.] Penalty parameter update.
				Set
				\begin{equation*}
					\beta_{k+1} = \left\{
					\begin{aligned}
						&\beta_k + \delta_\beta, \qquad &&\text{if}~\max\{\beta_k, 1/t^{k+1}\} < \|z^{k+1}-z^k\|^{-1}, \\
						&\beta_k, \qquad &&\text{otherwise}.
					\end{aligned}\right.
				\end{equation*}
				\item[5.] Set $k:=k+1$.
		\end{itemize}}
		\EndFor
	\end{algorithmic}
\end{algorithm}

In DCA2 of \cite{le2014dc}, the subproblem (\ref{subp})  was solved as a constrained optimization problem and  a Lagrange multiplier is used to update the penalty parameter. Since our penalty parameter update rule does not involve any multipliers, it is easier to implement. 
In the rest of this section we show that the proposed algorithm converges. Let us start with the following lemma which provides a sufficient decrease of the merit function of (DC) defined by   
$$\varphi_k(z):=g_0(z)-h_0(z)+\beta_k \max\{g_1(z)-h_1(z),0\}.$$
\begin{lemma}\label{suff_decreasenew}
	Let $\{z^k\}$ be a sequence of iterates generated by  iP-DCA as defined in Algorithm~{\rm\ref{ipDCA}}. If the inexact criterion \eqref{inexact1}  or \eqref{inexact2} is applied, then
	$z^k$ satisfies 
	{
		\begin{eqnarray*}
			\varphi_k(z^k) &\ge &  \varphi_k(z^{k+1}) + \frac{\rho}{2} \|z^{k+1} - z^k\|^2 - \frac{1}{2\rho}\zeta_{k}^2, \label{suff_decrease_eq1} \\ 
			\mbox{or} \qquad \varphi_k(z^k) & \geq &  \varphi_k(z^{k+1}) + \frac{\rho}{2} \|z^{k+1} - z^k\|^2 - \frac{\rho}{4}\|z^k - z^{k-1}\|^2, \label{suff_decrease_eq2}
		\end{eqnarray*} where $\zeta_k\geq 0  \mbox{ satisfying } \sum_{k = 0}^{\infty} \zeta_{k}^2 < \infty$ respectively. }
\end{lemma}
\begin{proof}
	Since $z^{k+1}$ is an approximation solution to problem \eqref{subp} with inexact criterion \eqref{inexact1} or \eqref{inexact2}, there exists a vector $e_k$ such that $e_k \in \partial \tilde{\varphi}_k(z^{k+1}) + \mathcal{N}_\Sigma(z^{k+1})\subseteq \partial (\tilde{\varphi}_k+\delta_\Sigma )(z^{k+1})$ and 
	\begin{equation}\label{15-16}
		\|e_k\| \leq \zeta_k \mbox{ or } \quad \|e_k\| \leq \frac{\sqrt{2}}{2} \rho\|z^k - z^{k-1}\|
		,\end{equation} respectively.   As $\tilde{\varphi}_k$ is strongly convex with modulus $\rho$, $\Sigma$ is a closed convex set and $z^k\in \Sigma$, we have
	\begin{equation}\label{suff_decrease_proof_eq1}
		\begin{aligned}
			\tilde{\varphi}_k(z^{k}) &\ge \tilde{\varphi}_k(z^{k+1}) + \langle e_k, z^{k+1} - z^k \rangle + \frac{\rho}{2}\|z^{k+1} - z^k\|^2 \\
			& \ge \tilde{\varphi}_k(z^{k+1}) - \frac{1}{2\rho}\|e_k\|^2 - \frac{\rho}{2}\|z^{k+1} - z^k\|^2 + \frac{\rho}{2}\|z^{k+1} - z^k\|^2 \\
			& = \tilde{\varphi}_k(z^{k+1}) - \frac{1}{2\rho}\|e_k\|^2.
		\end{aligned}
	\end{equation}
	Next, by the convexity of $h_i(z)$ and $\xi^k_i\in \partial h_i(z^k)$, $i=0,1$, there holds that
	\[
	h_i(z^{k+1}) \ge h_i(z^k)+ \langle \xi^k_i, z^{k+1} - z^k \rangle, \quad i = 0,1,
	\]
	and thus
	$
	\tilde{\varphi}_k(z^{k+1}) \ge \varphi_k(z^{k+1}) + \frac{\rho}{2} \|z^{k+1} - z^k\|^2.
	$
	Combined with \eqref{suff_decrease_proof_eq1}, we have
	\[
	\varphi_k(z^k) = \tilde{\varphi}_k(z^{k}) \ge \tilde{\varphi}_k(z^{k+1})  - \frac{1}{2\rho}\|e_k\|^2 \ge \varphi_k(z^{k+1})  - \frac{1}{2\rho}\|e_k\|^2 + \frac{\rho}{2} \|z^{k+1} - z^k\|^2.
	\]
	The conclusion follows immediately from (\ref{15-16}). 
	\qed
\end{proof}

The following theorem is the main result of this section. It proves that any accumulation point of  iP-DCA is a KKT point as long as the penalty parameter sequence $\{\beta_k\}$ is bounded.

\begin{theorem}\label{thm1}
	Suppose $f_0$ is bounded below on $\Sigma$ and the sequences $\{z^k\}$ and $\{\beta_k\}$ generated by   iP-DCA are bounded. Moreover suppose functions $g_0$, $g_1$, $h_1$, $h_0$ are locally Lipschitz on set $\Sigma$. Then  every accumulation point of $\{z^k\}$ is a KKT point of problem (DC).
\end{theorem}
\begin{proof}
	Since $\{\beta_k\}$ is bounded, there exists some iteration index $k_0$ such that
	$
	\beta_k = \beta_{k_0}, \quad \forall k \ge k_0,
	$
	and thus $\varphi_k(z) = \varphi_{k_0}(z)$ for all $k \ge k_0$.
	As $f_0$ is bounded below, $\varphi_k(z)$ is bounded below for all $k \ge k_0$. 
	Then, by the inequality \eqref{suff_decrease_eq1} and \eqref{suff_decrease_eq2} obtained in Lemma \ref{suff_decreasenew}, we have
	$$
	\sum_{k = 1}^{\infty} \|z^{k+1}-z^{k}\|^2 < + \infty, \qquad \lim_{k \rightarrow \infty} \|z^{k+1}-z^{k}\| = 0, 
	$$
	and thus $\beta_k < \|z^{k+1}-z^k\|^{-1}$ always holds when $k$ is large enough.
	According to the update strategy of $\beta_k$ in  iP-DCA, there exists some iteration index $k_1$ such that 
	\[
	t^{k+1} := \max\{g_1(z^{k+1}) -h_1(z^k)- \langle\xi_1^k, z^{k+1} - z^k\rangle, 0\} \le \|z^{k+1}-z^k\| \qquad \forall k \ge k_1,
	\]
	and thus $t^{k} \rightarrow 0$.
	Since $z^{k+1}$ is an approximate solution to problem \eqref{subp} and inexact criterion \eqref{inexact1} or \eqref{inexact2} holds, there exists a vector $e_k$ such that $e_k \in \partial \tilde{\varphi}_k(z^{k+1}) + \mathcal{N}_\Sigma(z^{k+1})$ and (\ref{15-16}) holds. According to the sum rule (see, e.g., \cite[Theorem 23.8]{rockafellar}\cite[Corollary 1 to Theorem 2.9.8]{clarke1990optimization}) and the subdifferential calculus rules for the pointwise maximum (see, e.g., \cite[Proposition 2.3.12]{clarke1990optimization}), there exist $\tilde{\lambda}^{k+1} \in [0,1]$ and $\eta_i^{k+1} \in \partial g_i(z^{k+1}) (i=0,1)$ such that 
	\begin{eqnarray} 
		&& e_k \in \eta_0^{k+1}  - \xi_0^k + \beta_k \tilde{\lambda}^{k+1} (\eta_1^{k+1}  -\xi_1^k) +\rho(z^{k+1}-z^k) +{\cal N}_\Sigma(z^{k+1}) ,\label{optimalityc}\\
		&& g_1(z^{k+1}) -h_1(z^k)- \langle\xi_1^k, z^{k+1}-z^k\rangle \leq t^{k+1}, \label{eqn10} \\
		&& \tilde{\lambda}^{k+1}(g_1(z^{k+1}) -h_1(z^k)- \langle\xi_1^k, z^{k+1}-z^k\rangle- t^{k+1})=0, \label{eqn11} 
		\\
		&& t^{k+1}( 1 -\tilde{\lambda}^{k+1}) = 0 , \quad t^{k+1}\geq 0. 
		\label{eqn9}
	\end{eqnarray}
	Since $\{\beta_k\tilde{\lambda}^{k+1}\}$ is bounded, we may suppose that  $\tilde{z}$ and $\tilde{\lambda}$ are  accumulation points of $\{z^k\}$ and $\{\beta_k\tilde{\lambda}^{k+1}\}$ respectively. Taking subsequences if necessary, without loss of generality we may assume that $z^k \rightarrow \tilde z\in \Sigma$ and $\beta_k\tilde{\lambda}^{k+1} \rightarrow \tilde \lambda$.
	Now passing onto the limit as $k  \rightarrow \infty$ in \eqref{optimalityc}-\eqref{eqn11}, as $e_k \rightarrow 0$ from $\zeta_{k} \rightarrow 0$ in \eqref{inexact1} or $\|z^{k+1}-z^{k}\| \rightarrow 0$ in \eqref{inexact2} and $t^k \rightarrow 0$, since $g_i(z)$, $h_i(z)$, $i = 0,1$ are locally Lipschitz continuous at $\tilde{z}$,  $\partial g_i(z)$, $\partial h_i(z)$, $i = 0,1$ and $\mathcal{N}_\Sigma(z)$ are outer semicontinuous, we obtain that
	$\tilde{z}$ is a KKT solution of problem (DC). \qed
\end{proof}

Notice that the boundedness of the penalty parameters is needed for an accumulation
point to be a KKT point. The following proposition provides a sufficient
condition for the boundedness of the penalty parameters sequence $\{\beta_k\}$.

\begin{proposition}\label{Prop1_bounded_beta}
	Suppose that the iterate sequence $\{z^k\}$ generated by iP-DCA is bounded. Moreover suppose functions $g_0,g_1,h_1,h_0$ are Lipschitz around at any accumulation point of   $\{z^k\}$. Assume that ENNAMCQ holds at any accumulation points of the  sequence  $\{z^k\}$. Then the sequence $\{\beta_k\}$ must be bounded.
\end{proposition}
\begin{proof}
	The proof is inspired by \cite[Theorem 3.1]{le2014dc}. 
	To the contrary, suppose that $\beta_k \rightarrow + \infty$ as $k \rightarrow \infty$. Then there exist infinitely many indices $j$ such that
	\[
	\beta_{k_j}< \|z^{k_j+1}- z^{k_j}\|^{-1}  \mbox{ and } \quad t^{k_j+1} > \|z^{k_j+1}- z^{k_j}\|,
	\]
	and thus 
	$$
	\lim_{j \rightarrow \infty} \|z^{k_j+1}-z^{k_j}\| = 0, \qquad t^{k_j+1} > 0, \quad \forall j.
	$$
	From \eqref{eqn9}, since $t^{k_j+1} > 0$ for all $j$, we have $\tilde{\lambda}^{k_j+1} = 1$ for all $j$ and thus $\lambda^{k_j+1}:= \beta_{k_j}\tilde{\lambda}^{k_j+1} \rightarrow +\infty$ as $j \rightarrow \infty$.
	Taking a further subsequence, if necessary, we can assume that $z^{k_j} \rightarrow \tilde{z}\in \Sigma$ as $j \rightarrow \infty$.
	If $g_1 (\tilde{z})-h_1 (\tilde{z}) < 0$, then as $g_1,h_1$ are continuous at $\tilde{z}$, $\{\xi^{k_j}\}$ is bounded, and $\lim_{j \rightarrow \infty} \|z^{k_j+1}-z^{k_j}\| = 0$, when $j$ is sufficiently large, one has
	$
	g_1(z^{k_j+1}) -h_1(z^{k_j})- \langle \xi^{k_j}_1, z^{k_j+1} - z^{k_j} \rangle  < 0,
	$
	which contradicts to $t^{k_j+1}:= \max\{g_1(z^{k_j+1}) -h_1(z^{k_j})- \langle \xi^{k_j}_1, z^{k_j+1} - z^{k_j}\rangle,0\} > 0$ for all $j$.
	Thus, $g_1 (\tilde{z})-h_1 (\tilde{z})  \ge 0$. 
	From \eqref{optimalityc}, we have 
	\[
	\begin{aligned}
			e_{k_j} \in &\partial g_0(z^{k_j+1})- \partial h_0(z^{k_j})+ \lambda^{k_j+1} \partial g_1(z^{k_j+1}) -\lambda^{k_j+1} \partial h_1(z^{k_j})\\&  + \rho(z^{k_j+1} - z^{k_j}  ) + {N}_\Sigma(z^{k_j+1}), 
	\end{aligned}
	\]
	where $\lambda^{k+1}:= \beta_k\tilde{\lambda}^{k+1}$.
	Dividing both sides of this equality by $ \lambda^{k_j+1}$, and passing onto the limit as $j \rightarrow \infty$, 
	we have
	$
	0 \in \partial g_1(\tilde{z})-\partial h_1(\tilde{z}) + \mathcal{N}_\Sigma(\tilde{z}),
	$
	which contradicts ENNAMCQ. \qed
\end{proof}

\subsection{Inexact proximal linearized DCA with simplified penalty parameter update}

Recall that  iP-DCA defined in Algorithm~{\rm\ref{ipDCA}} requires minimization of a strongly convex subproblem (\ref{subp}).
In this subsection, we assume that $g_1$ is $L$-smooth which means that $\nabla g_1(z)$ is Lipschitz continuous with constant $L$ on $\Sigma$. 
This setting motivates a very simple linearization approach inspired by the idea behind the proximal gradient method (see \cite{FOMbook} and the references therein). Specifically, we 
linearize both the concave part and the convex smooth part of the DC structure. Such a linearization approach makes subproblems easier to solve compared to iP-DCA.
Given a current iterate $z^k\in \Sigma$ with  $k=0,1,\ldots$, we select {a subgradient} $\xi^k_i\in \partial h_i(z^k)$, for $i=0,1$. 
Then we solve the following subproblem approximately.
\begin{equation} \label{subplin} 
	\begin{aligned}
		&\underset{z \in \Sigma}{\text{min}} ~~    \hat{\varphi}_k(z) := \,\, g_0(z) - h_0(z^k) -\langle \xi_0^k, z-{z^k}\rangle  +\frac{\rho_k}{2} \|z-z^k\|^2 \\
		& \qquad \qquad \quad +\beta_k \max\{g_1(z^k) +\langle \nabla g_1(z^k), z-z^k\rangle -h_1(z^k)- \langle\xi_1^k, z-z^k\rangle, 0\},
	\end{aligned}
\end{equation}
where $\rho_k$ and $\beta_k$ are the adaptive proximal and penalty parameters respectively.
Choose $z^{k+1}$ as an {\em approximate minimizer} of the convex subproblem (\ref{subplin}) satisfying one of the following two inexact criteria
\begin{equation}\label{inexact1_lin}
	\mathrm{dist}(0, \partial \hat{\varphi}_k(z^{k+1}) + \mathcal{N}_\Sigma(z^{k+1})) \le \zeta_{k}, \qquad \mbox{ for some } \zeta_k \mbox{ satisfying } \sum_{k = 0}^{\infty} \zeta_{k}^2 < \infty,
\end{equation}
and
\begin{equation}\label{inexact2_lin}
	\mathrm{dist}(0, \partial \hat{\varphi}_k(z^{k+1}) + \mathcal{N}_\Sigma(z^{k+1})) \le \frac{\sqrt{2}}{2}\sigma\|z^k - z^{k-1}\|.
\end{equation}
This yields the inexact proximal linearized DCA (iPL-DCA),
whose exact description is given in Algorithm \ref{iplDCA}.

\begin{algorithm}[h]
	\caption{iPL-DCA}\label{iplDCA}
	\begin{algorithmic}[1]
		\State Take an initial point $z^0\in \Sigma$; $\delta_\beta > 0$, $\sigma >0$, an initial penalty parameter $\beta_0>0$, an initial regularizer parameter $\rho_0 = \frac{1}{2}\beta_0L+\sigma$, $tol>0$.
		\For{$k=0,1,\ldots$}{
			\begin{itemize}
				\item[1.] Compute $\xi^k_i\in \partial h_i(z^k)$, $i=0,1$.
				\item[2.] Obtain an inexact solution $z^{k+1}$ of (\ref{subplin}) satisfying  (\ref{inexact1_lin}) or (\ref{inexact2_lin}).
				\item[3.] Stopping test. Compute $t^{k+1} := \max\{g_1(z^k) +\langle \nabla g_1(z^k), z^{k+1}-z^k\rangle -h_1(z^k)- \langle\xi_1^k, z^{k+1} - z^k\rangle, 0\}$. Stop if $\max \{ \|z^{k+1}-z^k\|, t^{k+1}\} <tol$.
				\item[4.] Penalty parameter update.
				Set
				\begin{eqnarray*}
					\beta_{k+1} &=& \left\{
					\begin{aligned}
						&\beta_k + \delta_\beta, \qquad &&\text{if}~ \max\{\beta_k, 1/t^{k+1}\} < \|z^{k+1}-z^k\|^{-1}, \\
						&\beta_k, \qquad &&\text{otherwies}.
					\end{aligned}\right.\\
					\rho_{k+1}&=& \frac{1}{2}\beta_{k+1} L +\sigma.
				\end{eqnarray*}
				\item[5.] Set $k:=k+1$.
		\end{itemize}}
		\EndFor
	\end{algorithmic}
\end{algorithm}

Recall that the merit function of (DC) is defined by   
$\varphi_k(z):=g_0(z)-h_0(z)+\beta_k \max\{g_1(z)-h_1(z),0\}.$
Similar to Lemma \ref{suff_decreasenew}, we first give following sufficiently decrease result of  iPL-DCA.
\begin{lemma}\label{lin_suff_decrease}
	Let $\{z^k\}$ be the sequence of iterates generated by  iPL-DCA  as defined in Algorithm~{\rm\ref{iplDCA}}. If the inexact criterion \eqref{inexact1_lin} or \eqref{inexact2_lin} is applied, then 
	$z^k$ satisfies 
	\begin{eqnarray*}
		\varphi_k(z^k)  & \ge & \varphi_k(z^{k+1}) + \frac{\sigma}{2} \|z^{k+1} - z^k\|^2 - \frac{1}{2\sigma}\zeta_{k}^2, \label{lin_suff_decrease_eq1} \\
		\varphi_k(z^k) & \ge &    \varphi_k(z^{k+1}) + \frac{\sigma}{2} \|z^{k+1} - z^k\|^2 - \frac{\sigma}{4}\|z^k - z^{k-1}\|^2, \label{lin_suff_decrease_eq2}
	\end{eqnarray*} respectively.
\end{lemma}
\begin{proof}
	Since $z^{k+1}$ is an approximation solution to problem \eqref{subplin} with inexact criterion \eqref{inexact1_lin} or \eqref{inexact2_lin}, there exists a vector $e_k$ such that $e_k \in \partial \hat{\varphi}_k(z^{k+1}) + \mathcal{N}_\Sigma(z^{k+1}))\subseteq \partial (\tilde{\varphi}_k+\delta_\Sigma )(z^{k+1})$ and  \begin{equation}\label{15-16new}
		\|e_k\| \leq \zeta_k \mbox{ or } \quad \|e_k\| \leq \frac{\sqrt{2}}{2} \sigma \|z^k - z^{k-1}\|
		,\end{equation} respectively.  As $\hat{\varphi}_k$ is strongly convex with modulus $\rho_k$ and $\Sigma$ is a closed convex set, we have
	\begin{equation}\label{lin_suff_decrease_proof_eq1}
		\begin{aligned}
			\hat{\varphi}_k(z^{k}) &\ge \hat{\varphi}_k(z^{k+1}) + \langle e_k, z^{k+1} - z^k \rangle + \frac{\rho_k}{2}\|z^{k+1} - z^k\|^2 \\
			& \ge \hat{\varphi}_k(z^{k+1}) - \frac{1}{2\sigma}\|e_k\|^2 - \frac{\sigma}{2}\|z^{k+1} - z^k\|^2 + \frac{\rho_k}{2}\|z^{k+1} - z^k\|^2 \\
			& = \hat{\varphi}_k(z^{k+1}) - \frac{1}{2\sigma}\|e_k\|^2 + \frac{\rho_k - \sigma}{2}\|z^{k+1} - z^k\|^2.
		\end{aligned}
	\end{equation}
	Next, by the convexity of $h_i(z)$ and $\xi^k_i\in \partial h_i(z^k)$, $i=0,1$, we have
	\[
	h_i(z^{k+1}) \ge h_i(z^k)+ \langle \xi^k_i, z^{k+1} - z^k \rangle, \quad i = 0,1.
	\]
	And since $g_1$ is $L$-smooth, we have
	\[
	g_1(z^{k+1}) \le g_1(z^k) +\langle \nabla g_1(z^k), z-z^k\rangle + \frac{L}{2}\|z^{k+1} - z^k\|^2.
	\]
	Thus, we have
	\[
	\hat{\varphi}_k(z^{k+1}) \ge \varphi_k(z^{k+1}) + \frac{\rho_k - \beta_kL}{2} \|z^{k+1} - z^k\|^2.
	\]
	Combined with \eqref{lin_suff_decrease_proof_eq1}, we have
	\[
	\begin{aligned}
		\varphi_k(z^k) = \hat{\varphi}_k(z^{k}) &\ge \hat{\varphi}_k(z^{k+1})  - \frac{1}{2\sigma}\|e_k\|^2 + \frac{\rho_k - \sigma}{2}\|z^{k+1} - z^k\|^2 \\
		&\ge \varphi_k(z^{k+1})  - \frac{1}{2\sigma}\|e_k\|^2 + \frac{2\rho_k - \beta_kL - \sigma}{2} \|z^{k+1} - z^k\|^2 \\
		& =\varphi_k(z^{k+1})  - \frac{1}{2\sigma}\|e_k\|^2 + \frac{\sigma}{2} \|z^{k+1} - z^k\|^2.
	\end{aligned}
	\]
	Then the conclusion follows immediately from (\ref{15-16new}).
	\qed
\end{proof}

Similar to Theorem \ref{thm1} and Proposition \ref{Prop1_bounded_beta}, by Lemma \ref{lin_suff_decrease}, the following  convergence results of  iPL-DCA can be derived easily.
The proofs are purely technical and thus omitted.

\begin{theorem}\label{convergence_thm}
	Suppose $f_0$ is bounded below and the sequences $\{z^k\}$ and $\{\beta_k\}$ generated by iPL-DCA are bounded,  functions $g_0,h_1,h_0$ are  locally Lipschitz on set $\Sigma$. Then every accumulation point of $\{z^k\}$ is a KKT point for problem (DC).
\end{theorem}

\begin{proposition}\label{lin_Prop_bounded_beta}
	Suppose the sequence $\{z^k\}$ generated by  iPL-DCA is bounded, functions $g_0,h_1,h_0$ are Lipschitz around at any accumulation point of   $\{z^k\}$,  and ENNAMCQ holds at any accumulation points of the sequence $\{z^k\}$. Then the sequence $\{\beta_k\}$ is bounded.
\end{proposition}

\begin{remark}
	In fact, if $g_0$ is further assumed to be differentiable and $\nabla g_0$ is Lipschitz continuous, we can also linearize  $g_0$ in  iPL-DCA. The proof of convergence is similar.
\end{remark}

\section{DC algorithms for solving DCBP}

In this section
we will show how to solve problem (DCBP) numerically. 
It is obvious that  problem $({\rm VP})_\epsilon$ is problem (DC) with 
$$z:=(x,y), \ f_0(x,y):=F_1(x,y)-F_2(x,y), \  f_1(x,y):= f(x,y)-v(x)-\epsilon, \ \Sigma =C.$$
According to \cite[Theorem 10.4]{rockafellar}, since $F_1(x,y), F_2(x,y), f(x,y) $ are Lipschitz continuous near every point on an open convex set containing $C$ and hence  Lipschitz continuous near every point on $C$.  However our problem $({\rm VP})_\epsilon$ involves the value function which is an extended-value function $v(x): X \rightarrow [-\infty, \infty]$ defined by
$$v(x) := \inf_{y \in Y} \left\{ f(x,y) ~ s.t.~ g(x,y) \le 0\right\},$$
with the convention of $v(x)=+\infty$ if the feasible region ${\cal F}(x)$ is empty.
To apply the proposed DC algorithms, we need to answer the following questions.
\begin{itemize}
	\item[(a)] Is the value function  convex and  locally Lipschitz  on the convex set $X$ and how to obtain one element from $\partial v(x^k)$ in terms of problem data?
	\item[(b)] Will the constraint qualification ENNAMCQ hold at any accumulation point of the iterate sequence?
\end{itemize}
We now give answers to  these questions in the next two subsections.

\subsection{Lipschitz continuity and  the subdifferential of the value function} 

Thanks to the full convex structure of the lower level problem in (DCBP), the value function turns out to be convex and  Lipschitz continuous under our problem setting as shown below.
\begin{lemma}\label{Prop1}
	The value function $v(x):X \rightarrow \mathbb{R}$ is convex and  Lipschitz continuous around any point   in  set $X$.  Given $\bar{x}\in X$ and $\bar{y}\in {S}(\bar{x})$, we have
	\begin{equation}
		\partial v(\bar{x}) = \{ \xi \in \mathbb{R}^n : (\xi,0) \in \partial \phi(\bar{x}, \bar{y}) \},\label{gvaluefunction}
	\end{equation}
	where $\phi(x,y) :=  f(x,y) + \delta_D(x,y)$, $D:=\{(x,y)\in {\cal O} \times Y \mid g(x,y)\leq 0\},$ 
	with ${\cal O}$ being the open set defined in the introduction.
\end{lemma}
\begin{proof}
	First we extend the definition of the value function from any element $x\in X$ to the whole space $ \mathbb{R}^n$ as follows:
	$$v(x) := \inf_{y \in \mathbb{R}^m} \phi(x,y), \qquad \forall x\in \mathbb{R}^n.$$
	It follows that $v(x)=+\infty$ for $x\not \in {\cal O}$.
	In our problem setting, $f$ is fully convex on an open convex set containing  the convex set	$C$ and hence we can assume without loss of generality that $f$ is fully convex on  the convex set $D$. Therefore the extended-valued function $\phi(x,y): \mathbb{R}^n\times \mathbb{R}^m\rightarrow [-\infty, \infty]$ is  convex. 
                                                                                                                                                       	The convexity of the value function $v(x) = \inf_{y \in \mathbb{R}^m} \phi(x,y)$ then follows from \cite[Theorem 1]{rockafellar1974conjugate}.  Hence the value function restricted on set $X$ is convex.  Next, according to \cite[Theorem 24]{rockafellar1974conjugate}, {we have the equation (\ref{gvaluefunction}).}  By assumption stated in the introduction of the paper,  the feasible region of the lower level program ${\cal F}(x):=\{y\in Y: g(x,y)\leq 0\}\not =\emptyset$  and $v(x)\not =-\infty$ for all  $x$ in  the open set ${\cal O}$. 
	Hence $v(x): \mathbb{R}^n \rightarrow [-\infty,\infty]$ is proper convex.
	Since   dom$v:=\{x: v(x)<+\infty\}= \{x: {\cal F}(x)\not =\emptyset\}\supseteq {\cal O}\supseteq X$,  we have $X \subseteq $  int(dom$v$). The result on  Lipschitz continuity of the value function follows from \cite[Theorem 10.4]{rockafellar}. \qed
\end{proof}
By using some sensitivity analysis techniques,  {a subgradient} of the value function $v(x)$ can be expressed in terms of Lagrangian multipliers. 
In particular, given $\bar{y}\in {S}(\bar{x})$, we denote the set of KKT multipliers of the lower-level problem $(P_{\bar x})$  by 
\[
\begin{aligned}
&KT(\bar{x},\bar{y})\\
 := &\left \{ \gamma \in \mathbb{R}^l_+ \Big|  0 \in \partial_y f(\bar{x}, \bar{y}) + \sum_{i = 1}^{l} \gamma_i \partial_y g_i(\bar{x}, \bar{y}) + \mathcal{N}_{Y}(\bar{y}), \quad \sum_{i = 1}^{l}\gamma_i g_i(\bar{x},\bar{y}) = 0   \right \}.
\end{aligned}
\]


\begin{theorem}\label{Thm4.1}
	Let $\bar{x}\in X$ and $\bar{y} \in {S}(\bar{x})$. Then
	\begin{eqnarray}\label{inclusionvaluef}
		\lefteqn{\partial v(\bar x)  \supseteq} \nonumber \\
			&& \Bigg \{\xi \Big| (\xi,0)\in   \partial f(\bar x,\bar y)+ \sum_{i = 1}^{l} \gamma_i \partial g_i(\bar{x}, \bar{y}) + \{0\} \times \mathcal{N}_{ Y}(\bar{y}),\nonumber \\
			&& \hspace*{140pt} \gamma \in \mathbb{R}^l, \gamma \ge 0, \sum_{i = 1}^{l}\gamma_i g_i(\bar{x},\bar{y}) = 0  \Bigg \},
	\end{eqnarray} 	
	and the equality holds in (\ref{inclusionvaluef}) provided that 
	\begin{equation} \mathcal{N}_E(\bar{x}, \bar{y}) = \Big\{ \sum_{i = 1}^{l} \gamma_i \partial g_i(\bar{x}, \bar{y}) + \{0\} \times \mathcal{N}_{ Y}(\bar{y}) \mid \gamma\ge 0, \sum_{i = 1}^{l}\gamma_i g_i(\bar{x},\bar{y}) = 0 \Big\}, \label{normalconeeq}\end{equation}
	where $E:=\{(x,y)\in \mathbb{R}^n \times Y: g(x,y)\leq 0\}$. 
	
	Moreover if  the partial derivative formula holds
	\begin{eqnarray}
		\partial f(\bar x,\bar y)= \partial_x f(\bar x,\bar y)\times \partial_y f(\bar x,\bar y) ,&& 
		\partial g_i(\bar x,\bar y)= \partial_x g_i(\bar x,\bar y)\times \partial_y g_i(\bar x,\bar y) \label{partialdnew}
	\end{eqnarray}
	then 
	\begin{equation}\label{valuefinc}
		\bigcup_{ \gamma \in KT(\bar{x},\bar{y})} \left ( \partial_x f(\bar{x}, \bar{y}) + \sum_{i = 1}^{l} \gamma_i \partial_x g_i(\bar{x}, \bar{y})  \right )\subseteq \partial v(\bar{x}),
	\end{equation}
	and the equality in  (\ref{valuefinc}) holds provided that 
	(\ref{normalconeeq}) holds.
\end{theorem}
\begin{proof} Let $\phi_E(x,y):=f(x,y)+\delta_E(x,y)=f(x,y)+\delta_{ Y} (y)+\sum_{i=1}^l \delta_{C_i}(x,y)$ with $C_i:={\{(x,y)|g_i(x,y)\leq 0\}}$. Then 
	by the sum rule (see, e.g., \cite[Theorem 23.8]{rockafellar}\cite[Corollary 1 to Theorem 2.9.8]{clarke1990optimization} ) and the fact that $\mathcal{N}_E=\partial \delta_E$,	
	we have
	\begin{equation}\label{sumrule}
		\partial f(\bar x,\bar y)+\{0\} \times \mathcal{N}_{ Y}(\bar y) +\sum_{i=1}^l\mathcal{N}_{C_i}(\bar x,\bar y)\subseteq\partial\phi_E(\bar x,\bar y)  .\end{equation}
	When $g_i(\bar x,\bar y)<0$, we have $(\bar x,\bar y)\in {\rm int }C_i$ and hence $\gamma_i=0\in 
	\mathcal{N}_{C_i}(\bar x,\bar y)$. Otherwise if $g_i(\bar x,\bar y)=0$, by definition of subdifferential and the normal cone we can show that for any $\gamma_i\geq 0$,
	$\gamma_i \partial g_i(\bar x,\bar y)\subseteq \mathcal{N}_{C_i}(\bar x,\bar y) .$
	Hence together with (\ref{sumrule}), we have  
	$$
	\begin{aligned}
		&\Big\{\partial f(\bar x,\bar y)+ \sum_{i = 1}^{l} \gamma_i \partial g_i(\bar{x}, \bar{y}) +\{0\} \times \mathcal{N}_{ Y}(\bar y) \mid \gamma \in \mathbb{R}^l, \gamma \ge 0, \sum_{i = 1}^{l}\gamma_i g_i(\bar{x},\bar{y}) = 0 \Big\} \\
		 &  \subseteq\, \partial  \phi_E(\bar x,\bar y) .
	\end{aligned}
	$$ 
	Since $\partial \phi_E(\bar x,\bar y)=\partial \phi(\bar x,\bar y)$ where $\phi(x,y)=f(x,y)+\delta_D(x,y)$ {with $D:=\{(x,y)\in {\cal O} \times Y \mid g(x,y)\leq 0\}$,} it follows  from Lemma \ref{Prop1} that 
	\begin{eqnarray*}
	\lefteqn{\partial v(\bar x)= \left \{\xi\mid (\xi, 0) \in \partial \phi(\bar x,\bar y) \right\}= \left \{\xi\mid (\xi, 0) \in \partial  \phi_E(\bar x,\bar y) \right\}}\\
		&\supseteq & \Bigg  \{\xi\mid (\xi,0)\in   \partial f(\bar x,\bar y)+ \sum_{i = 1}^{l} \gamma_i \partial g_i(\bar{x}, \bar{y}) + \{0\} \times \mathcal{N}_{ Y}(\bar y),\\
		&& \hspace*{170pt} \gamma \in \mathbb{R}^l, \gamma \ge 0, \sum_{i = 1}^{l}\gamma_i g_i(\bar{x},\bar{y}) = 0  \Bigg \}.
	\end{eqnarray*} 	
Hence (\ref{inclusionvaluef}) holds.
	Since $f$ is Lipschitz continuous at $(\bar{x}, \bar{y})$, by  the sum rule (see, e.g., \cite[Theorem 23.8]{rockafellar}\cite[Corollary 1 to Theorem 2.9.8]{clarke1990optimization} ), we have $\partial  \phi_E(\bar{x}, \bar{y}) =  \partial f(\bar{x}, \bar{y}) + \mathcal{N}_E(\bar{x}, \bar{y})$. Hence
	if (\ref{normalconeeq}) holds, then
	\begin{eqnarray*}
		\lefteqn{ \partial v(\bar x) =\left \{\xi\mid(\xi, 0) \in \partial \phi(\bar x,\bar y) \right \}= \left \{\xi\mid (\xi, 0) \in \partial  \phi_E(\bar x,\bar y) \right\}}\\
		&= & \left \{\xi\mid(\xi, 0) \in \partial f(\bar{x}, \bar{y}) + \mathcal{N}_E(\bar{x}, \bar{y})  \right \} \\
		&=&\Bigg \{\xi\mid (\xi,0)\in   \partial f(\bar x,\bar y)+ \sum_{i = 1}^{l} \gamma_i \partial g_i(\bar{x}, \bar{y}) + \{0\} \times \mathcal{N}_{ Y}(\bar y),
		\\
		&& \hspace*{200pt}  
		\gamma \ge 0, \sum_{i = 1}^{l}\gamma_i g_i(\bar{x},\bar{y}) = 0 \Bigg  \}.
	\end{eqnarray*} 
	This shows that  the equality holds in (\ref{inclusionvaluef}).
	
	Now suppose that  (\ref{partialdnew}) holds.
	Then for any $\gamma \in \mathbb{R}^l, \gamma \ge 0, \sum_{i = 1}^{l}\gamma_i g_i(\bar{x},\bar{y}) = 0$, by the sum rule (see, e.g., \cite[Theorem 23.8]{rockafellar}\cite[Corollary 1 to Theorem 2.9.8]{clarke1990optimization} ) we have 
	\begin{eqnarray*}
		&&\partial \phi_E(\bar x,\bar y) \\
		&\supseteq& \partial f(\bar{x}, \bar{y}) + \sum_{i = 1}^{l} \gamma_i \partial g_i(\bar{x}, \bar{y}) +\{0\} \times \mathcal{N}_{ Y}(\bar y)  \\
		&=&  \left\{\partial_x f(\bar{x}, \bar{y}) + \sum_{i = 1}^{l} \gamma_i \partial_x g_i(\bar{x}, \bar{y}) \right\} \times \left\{ \partial_y f(\bar{x}, \bar{y}) + \sum_{i = 1}^{l} \gamma_i \partial_y g_i(\bar{x}, \bar{y}) + \mathcal{N}_{Y}(\bar{y})\right\}.
	\end{eqnarray*}
	Combining with \eqref{inclusionvaluef}, we obtain (\ref{valuefinc}).
	Similarly when (\ref{normalconeeq}) holds, the equality holds in (\ref{valuefinc}). \qed
\end{proof}

By the description in (\ref{valuefinc}), $\partial v(\bar x)$ can be calculated as long as (\ref{normalconeeq})  is satisfied. We claim that (\ref{normalconeeq}) is a mild condition. In fact,  by convexity, there always holds the inclusion 
$$\mathcal{N}_E(\bar{x}, \bar{y}) \supseteq \Big\{ \sum_{i = 1}^{l} \gamma_i \partial g_i(\bar{x}, \bar{y}) + \{0\} \times \mathcal{N}_{ Y}(\bar y)  : \gamma\ge 0, \sum_{i = 1}^{l}\gamma_i g_i(\bar{x},\bar{y}) = 0 \Big\}.$$ 
By virtue of \cite[Theorem 4.1]{calmness_multifunction}, the reverse inclusion also follows under standard {constraint qualifications}. Some suffcient conditions for (\ref{normalconeeq}) are thus summarized in the following proposition. 
\begin{proposition}\label{Prop4.4}  Equation (\ref{normalconeeq})  holds provided that the set-valued map
	$$\Psi(\alpha):=\{ (x,y) \in \mathbb{R}^n \times Y: g(x,y)+\alpha \leq 0\}$$ 
	is calm at $(0, \bar x,\bar y)$, i.e., there exist $\kappa, \delta >0$ such that
	$$ {\rm dist}_E(x,y) \leq \kappa \|\max\{g(x,y),0\}\| \qquad \forall (x,y)\in B_\delta (\bar x,\bar y)\cap E;$$ in particular if one of the following conditions:
	\begin{itemize}
		\item[(a)] The linear constraint qualification holds:  $g(x,y)$  is  an affine mapping of $(x, y)$ and $ Y$  is convex polyhedral.
		\item[(b)] The Slater condition holds:  there exists a point $(x_0, y_0)\in \mathbb{R}^n\times Y$ such that $g(x_0, y_0)<0$.
	\end{itemize} 
\end{proposition}
\begin{proof}By virtue of \cite[Theorem 4.1]{calmness_multifunction}, the reverse inclusion $$\mathcal{N}_E(\bar{x}, \bar{y}) \subseteq \Big\{ \sum_{i = 1}^{l} \gamma_i \partial g_i(\bar{x}, \bar{y}) + \{0\} \times \mathcal{N}_{ Y}(\bar y)  : \gamma\ge 0, \sum_{i = 1}^{l}\gamma_i g_i(\bar{x},\bar{y}) = 0 \Big\}$$ 
	holds provided that the system $y \in Y, g(x,y)\leq 0$ is calm at $(0, \bar x, \bar y)$. It is well-known that (a) or (b) is a sufficient condition for calmness. \qed
\end{proof}


\subsection{Motivations for studying the approximate bilevel program}
There are three motivations to consider the approximate program $({\rm VP})_{\epsilon}$. 
First, as shown in \cite{LinXuYe}, the solutions of $({\rm VP})_{\epsilon}$ approximate a true solution of the original bilevel program (DCBP) as $\epsilon$
approaches zero. Second,
the proximity from a local minimizer of $({\rm VP})_{\epsilon}$ to the solution set of  (DCBP) can be controlled by adjusting the value of $\epsilon$. The third motivation is that the approximate program $({\rm VP})_{\epsilon}$ when $\epsilon>0$ would satisfy the required constraint qualification automatically. We present the second motivation in the following proposition.
\begin{proposition}
	Suppose  $\mathcal{S}^*$,  the solution set of problem $({\rm VP})$,   is compact and the value function $v(x)$ is continuous.  Then for any $\delta > 0$, there exists $\bar{\epsilon} > 0$ such that for any $\epsilon \in (0,\bar{\epsilon}]$, there exists $(x_{\epsilon},y_{\epsilon})$ which is  a local minimum of $\epsilon$-approximation problem $({\rm VP})_{\epsilon}$ with  $\mathrm{dist}((x_{\epsilon},y_{\epsilon}),\mathcal{S}^*) < \delta$.
\end{proposition}
\begin{proof}
	To the  contrary, assume  that there exist $\delta > 0$ and sequence $\{\epsilon^k\}$ with $\epsilon^k \downarrow 0$ as $k\rightarrow \infty$ such that there does not exist $(x, y)$ being a local minimum of $\epsilon^k$-approximation problem $({\rm VP})_{\epsilon_k}$ satisfying $\mathrm{dist}((x^k,y^k),\mathcal{S}^*) < \delta$ for all $k$. Consider a point $(\hat{x}^k, \hat{y}^k)$, which is a global minimum to the following problem 
	\begin{equation*}
		\begin{aligned}
			\min_{(x,y)\in C} ~~& F(x,y) \\
			s.t. ~~~~& f(x,y) - v(x) \le \epsilon, \\
			&\mathrm{dist}((x,y),\mathcal{S}^*) \le \delta.
		\end{aligned}
	\end{equation*}
	Then by assumption, it holds that $\mathrm{dist}((\hat{x}^k, \hat{y}^k),\Sigma^*) = \delta$ and $F(\hat{x}^k, \hat{y}^k) \le F^*$, where $F^*$ is the optimal value of problem $({\rm VP})$. As $\mathcal{S}^*$ is compact, sequence $\{(\hat{x}^k, \hat{y}^k)\}$ is bounded and we can assume without lost of generality that $(\hat{x}^k, \hat{y}^k) \rightarrow (\bar{x},\bar{y})$ as $k \rightarrow \infty$. Since the value function $v$ is continuous, by taking $k \rightarrow \infty$ in $(\hat{x}^k, \hat{y}^k)\in C$ and $f(\hat{x}^k, \hat{y}^k) - v(\hat{x}^k) \le \epsilon^k$, we obtain the feasibility of the limit point $(\bar x,\bar y)$ for  problem (VP). Next, by taking $k \rightarrow \infty$ in $\mathrm{dist}((\hat{x}^k, \hat{y}^k),\mathcal{S}^*) = \delta$ and $F(\hat{x}^k, \hat{y}^k) \le F^*$, we obtain that $\mathrm{dist}((\bar x,\bar y),\mathcal{S}^*) = \delta$ and $F(\bar x,\bar y) \le F^*$, a contradiction. \qed
\end{proof}

Before we clarify the third motivation, 
we define the concepts of  NNAMCQ for problem $({\rm VP})$ and ENNAMCQ for problem $({\rm VP})_\epsilon$ where $\epsilon \geq 0$.
\begin{definition}
	Let $(\bar{x},\bar{y})$ be a feasible solution to problem $({\rm VP})$. We say that  NNAMCQ holds at $(\bar{x},\bar{y})$ for problem $({\rm VP})$ if
	\begin{equation}\label{newNNAMCQ}
		0 \notin \partial f(\bar{x},\bar{y}) - \partial v(\bar{x})\times \{0\}+ \mathcal{N}_C(\bar{x},\bar{y}).
	\end{equation}
	Let $(\bar{x},\bar{y}) \in C$, we say that ENNAMCQ holds at $(\bar{x},\bar{y})$ for problem $(VP)_\epsilon$ if either $f(\bar x,\bar y) - v(\bar x) < \epsilon$ or $f(\bar x,\bar y) - v(\bar x) \ge \epsilon$ but (\ref{newNNAMCQ}) holds.
\end{definition}
\begin{proposition}\label{Prop4.3}Let $(\bar{x},\bar{y})$ be a feasible solution to problem $(VP)$. Suppose that  $v(x) $ is Lipschitz continuous around $\bar{x}$. Then NNAMCQ never holds at $(\bar{x},\bar{y})$.
\end{proposition} 
\begin{proof}
	By definition of the value function, we can never have $f(\bar x,\bar y) - v(\bar x) < 0$ and we always have 
	$f(\bar x,\bar y) - v(\bar x) = 0$. But since $(\bar{x},\bar{y})$ is a feasible solution to problem $({\rm VP})$, it is easy to see that $(\bar{x},\bar{y})$  must be a solution to the following problem 
	$$\min_{(x,y)\in C} \left \{ f(x,y)-v(x) \right \}.$$ But by the optimality condition we must have 
	$$0 \in \partial  f(\bar{x},\bar{y}) - \partial v(\bar{x})\times \{0\}+ \mathcal{N}_C(\bar{x},\bar{y}).$$  This means that  (\ref{newNNAMCQ}) would never hold. \qed
\end{proof}
Although the NNAMCQ never hold for (VP),  {for  $\epsilon > 0$,  ENNAMCQ is a standard constraint qualification for $({\rm VP})_\epsilon$;  see e.g.  \cite[Proposition 8]{LinXuYe}.  Moreover }thanks to the model structures, it 
holds automatically for $({\rm VP})_\epsilon$ if $\epsilon > 0$.  Hence, according to the DC theories established in the preceding section, powerful DCA can be employed to solve $({\rm VP})_\epsilon$.

\begin{proposition}\label{prop4.2}
	For any $(\bar{x},\bar{y}) \in C$, 
	problem $({\rm VP})_{\epsilon}$ with $\epsilon>0$  satisfies ENNAMCQ at $(\bar{x},\bar{y})$.
\end{proposition}
\begin{proof} If $f(x,y) - v(x) < \epsilon$ holds, then by definition, ENNAMCQ holds at $(\bar{x},\bar{y})$. Now
	suppose that $f(\bar{x},\bar{y}) - v(\bar{x}) \ge \epsilon$ and ENNAMCQ does not hold, i.e., 
	\[
	0 \in \partial f(\bar{x},\bar{y}) -  \partial v(\bar{x})\times \{0\}+ \mathcal{N}_C(\bar{x},\bar{y}).
	\]
	It follows from {the partial subdifferentiation formula }(\ref{partialsubg})   that
	\begin{equation}\label{NNAMCQ_proof_eq1}
		0 \in \begin{bmatrix}
			\partial_x f(\bar{x},\bar{y})  - \partial v(\bar{x}) \\
			\partial_y f(\bar{x},\bar{y})
		\end{bmatrix} + \mathcal{N}_C(\bar{x},\bar{y}).
	\end{equation}
	By
	(\ref{partialsubg})  we have
	$$\mathcal{N}_C(\bar{x},\bar{y})=  \partial \delta_C(\bar x,\bar y) \subseteq \partial_x \delta_C(\bar x,\bar y)\times \partial_y \delta_C(\bar x,\bar y)\subseteq  \mathbb{R}^n \times \mathcal{N}_{C(\bar{x})}(\bar{y}),$$
	where 
	$
	C(\bar{x}) := \{y \in Y\mid g_i(\bar{x},y) \le 0 , i = 1,\ldots,l \}.
	$
	Thus, it follows from \eqref{NNAMCQ_proof_eq1} that 
	\[
	0 \in \partial_y f(\bar{x},\bar{y}) + \mathcal{N}_{C(\bar{x})}(\bar{y}),
	\]
	which further implies that
	$
	\bar{y} \in \mathcal{S}(\bar{x}).
	$
	However, an obvious contradiction to the assumption that $f(\bar{x},\bar{y}) - v(\bar{x}) \ge \epsilon$ occurs and thus the desired conclusion follows immediately. \qed
\end{proof}

By using the definition of KKT points for (DC) in  Definition \ref{Defn3.1}, under the assumption that the value function is locally Lipschitz continiuous, we define  KKT points for problem $({\rm VP})_\epsilon$.

\begin{definition}
	We say a  point $(\bar{x},\bar{y})$ is a KKT point of  problem $({\rm VP})_\epsilon$ with $\epsilon \geq  0$ if there exists $\lambda\geq 0$ such that
	\begin{equation*}\label{ClarkKKT}
		\left\{ \quad \begin{aligned}
			&0 \in \partial F_1(\bar{x},\bar{y}) -\partial F_2(\bar{x},\bar{y})+ \lambda \partial f(\bar{x},\bar{y}) - \lambda \partial v(\bar{x}) \times \{0\}+ \mathcal{N}_C(\bar{x},\bar{y}), \\
			& f(\bar{x},\bar{y}) - v(\bar{x}) - \epsilon \le 0, \quad  {\lambda} \left(f(\bar{x},\bar{y}) - v(\bar{x}) - \epsilon  \right) = 0.
		\end{aligned}\right.
	\end{equation*}
\end{definition}
%
By virtue of Proposition \ref{prop4.2} and Theorem \ref{Defn3.1}, we have the following necessary optimality condition. Since the issue of constraint qualifications for problem (VP) is complicated and it is not the main concern in this paper, we refer the reader to discussions on this topic in \cite{KuangYe,Yebook}.

\begin{theorem}\label{convergence1}
	Let $(\bar{x},\bar{y})$ be a local optimal solution to  problem $({\rm VP})_\epsilon$ with $\epsilon \geq  0$. 
	Suppose 
	either $\epsilon >  0$ or $\epsilon = 0$ and a constraint qualification holds. Then $(\bar{x},\bar{y})$ is a KKT point of problem $({\rm VP})_\epsilon$.
\end{theorem}


\subsection{Inexact proximal DCA for solving $({\rm VP})_{\epsilon}$}

In this subsection we  implement the proposed DC algorithms in Section \ref{sec:DC} to solve $({\rm VP})_{\epsilon}$. To proceed, let us 
describe  iP-DCA to solve $({\rm VP})_{\epsilon}$. 
Given a current iterate $(x^k, y^k)$ for each $k=0,1,\ldots$, solving the lower level problem parameterized  by $x^k$
\begin{equation*}\label{LL-iter}
	\min_{y \in Y}   f(x^k,y), ~ s.t.~ g(x^k,y) \le 0
\end{equation*}
leads to a solution $\tilde{y}^k \in {S}(x^k)$ and a corresponding KKT multiplier $\gamma^k \in KT(x^k,\tilde{y}^k)$. Select 
\begin{eqnarray}
	\xi_0^k\in \partial F_2(x^k,y^k), \quad \xi_1^k \in \partial_x f(x^k,\tilde{y}^k) + \sum_{i = 1}^{l} \gamma^k_i \partial_x g_i(x^k,\tilde{y}^k). \label{subgvaluefunction}
\end{eqnarray}
Note that according to (\ref{valuefinc}) in Theorem \ref{Thm4.1}, if the partial derivative formula holds, then $\xi_1^k \in \partial_x f(x^k,\tilde{y}^k) + \sum_{i = 1}^{l} \gamma^k_i \partial_x g_i(x^k,\tilde{y}^k)$ is an element of the subdifferential $ \partial v(x^k)$. 
Compute $(x^{k+1},y^{k+1})$ as an approximate minimizer of the strongly convex subproblem for $({\rm VP})_{\epsilon}$ given by				
\begin{equation} \label{DCA2_subproblem}
	\begin{aligned}
		\underset{(x,y)\in C}{\text{min}} ~~   & \,\,  F_1(x,y) -\langle \xi_0^k, (x,y)\rangle + \frac{\rho}{2} \|(x,y) - (x^k,y^k)\|^2  \\
		& +\beta_k \max\{f(x,y) - f(x^k,\tilde{y}^k) - \langle \xi_1^{k}, x - x^k \rangle - \epsilon, 0\}, 
	\end{aligned}
\end{equation}
satisfying one of the two inexact criteria. Under the assumption that $KT(x^k,y)$ is nonempty for all $y\in S(x^k)$, the description of iP-DCA on $({\rm VP})_{\epsilon}$ with $\epsilon\geq 0$ now follows:


\begin{algorithm}[H]
	\caption{iP-DCA for solving $({\rm VP})_{\epsilon}$}\label{DCA2}
	\begin{algorithmic}[1]
		\State Take an initial point $(x^0, y^0) \in X \times Y$; $\delta_\beta > 0$; an initial penalty parameter $\beta_0 > 0$, $tol > 0$.
		\For{$k=0,1,\ldots$}{
			\begin{itemize}
				\item[1.]  Solve the lower level problem $P_{x^k}$ defined in (\ref{LL-iter}) and
				obtain $\tilde{y}^k \in {S}(x^k)$ and $\gamma^k \in KT(x^k,\tilde{y}^k)$.
				\item[2.] Compute $\xi^k_i$, $i=0,1$ according to (\ref{subgvaluefunction}).				
				\item[2.] Obtain an inexact solution $(x^{k+1},y^{k+1})$ of (\ref{DCA2_subproblem}).
				\item[3.] Stopping test. Compute $t^{k+1}= \max\{ f(x^{k+1},y^{k+1}) - f(x^k,\tilde{y}^k) - \langle \xi_1^{k}, x^{k+1} - x^k \rangle - \epsilon,0 \}.$ Stop if $\max\{ \|(x^{k+1},y^{k+1}) - (x^{k},y^{k})\|,t^{k+1} \} < tol$.
				\item[4.] Penalty parameter update.
				Set
				\begin{equation*}
					\beta_{k+1} = \left\{
					\begin{aligned}
						&\beta_k + \delta_\beta, \qquad &&\text{if}~\max\{\beta_k, 1/t^{k+1}\} < \|(x^{k+1},y^{k+1}) - (x^{k},y^{k})\|^{-1}, \\
						&\beta_k, \qquad &&\text{otherwise}.
					\end{aligned}\right.
				\end{equation*}
				\item[5.] Set $k:=k+1$.
		\end{itemize}}
		\EndFor
	\end{algorithmic}
\end{algorithm}


Thanks to Proposition \ref{prop4.2}, when $\epsilon>0$, problem $({\rm VP})_{\epsilon}$ satisfies ENNAMCQ automatically. Moreover, 
since the partial subgradient formula (\ref{partialdnew}) holds, according to (\ref{valuefinc}) in Theorem \ref{Thm4.1},  the selection criteria in (\ref{subgvaluefunction}) implies that $\xi_1^k \in \partial v(x^k)$. 
Hence the convergence of iP-DCA for solving $({\rm VP})_{\epsilon}$ (Algorithm \ref{DCA2}) follows from Theorem \ref{thm1} and Proposition \ref{Prop1_bounded_beta}.

\begin{theorem}\label{convergence2}
	Assume that $F$ is bounded below on $C$.
	Let  $\{(x^k,y^k)\}$ be an iterate sequence generated by Algorithm \ref{DCA2}. Moreover assume that  the partial subgradient formula (\ref{partialdnew1}) holds at every iterate point $(x^k,y^k)$ and  $KT(x^k,y) \not =\emptyset$ for all $y\in S(x^k)$. Suppose that either $\epsilon >0$ or  $\epsilon = 0$ and the penalty sequence $\{\beta_k\}$ is bounded. Than any  accumulation point of $\{(x^k,y^k)\}$  is a KKT point of problem $({\rm VP})_{\epsilon}$.
\end{theorem}

We now assume 
that  the lower level objective $f$ is differentiable and $\nabla f$ is Lipschitz continuous with modulus $L_f$ on set $C$. Given a current iterate $(x^k, y^k)$, the next iterate 
$(x^{k+1},y^{k+1})$ can be returned as an approximate minimizer of subproblem (\ref{DCA2_subproblem}) with linearized $f$ given by	
\begin{equation}\label{lin_DCA2_subproblem}
	\begin{aligned}
		&\underset{(x,y)\in C}{\text{min}} ~~   F_1(x,y) -\langle \xi_0^k, (x,y)\rangle + \frac{\rho_k}{2} \|(x,y) - (x^k,y^k)\|^2 \\
		&  \quad  +\beta_k \max\{f(x^{k},y^{k}) + \langle \nabla f(x^{k},y^{k}), (x,y) -(x^{k},y^{k}) \rangle \\
		& \hspace*{180pt}  - f(x^k,\tilde{y}^k) - \langle \xi_1^{k}, x - x^k \rangle - \epsilon, 0\}.
	\end{aligned}
\end{equation}
Under the assumption that $KT(x^k,y) \not =\emptyset$ for all $y\in S(x^k)$, the iterate scheme of iPL-DCA on $({\rm VP})_{\epsilon}$ with $\epsilon\geq 0$ thus reads as:

\begin{algorithm}[H]
	\caption{iPL-DCA for solving $({\rm VP})_{\epsilon}$}\label{lin_DCA2}
	\begin{algorithmic}[1]
		\State Take an initial point $(x^0, y^0) \in X \times Y$;  $\delta_\beta, \sigma > 0$; $\beta_0 > 0$; $\rho_0 = \frac{1}{2}\beta_0L_f + \sigma$; $tol > 0$.
		\For{$k=0,1,\ldots$}{
			\begin{itemize}
				\item[1.]  Solve the lower level problem $P_{x^k}$ defined in (\ref{LL-iter}) and
				obtain $\tilde{y}^k \in {S}(x^k)$ and $\gamma^k \in KT(x^k,\tilde{y}^k)$.
				\item[2.] Compute $\xi^k_i$, $i=0,1$ according to (\ref{subgvaluefunction}).				
				\item[2.] Obtain an inexact solution $(x^{k+1},y^{k+1})$ of (\ref{lin_DCA2_subproblem}).
				\item[3.] Stopping test. Compute $t^{k+1} = \max\{ f(x^{k},y^{k}) + \langle \nabla f(x^{k},y^{k}), (x^{k+1},y^{k+1}) -(x^{k},y^{k}) \rangle - f(x^k,\tilde{y}^k) - \langle \xi_1^{k}, x^{k+1} - x^k \rangle - \epsilon,0 \}.$ Stop if $\max\{ \|(x^{k+1},y^{k+1}) - (x^{k},y^{k})\|,t^{k+1} \} < tol$.
				\item[4.] Penalty parameter update.
				Set
				\begin{equation*}\label{lin_beta_update}
					\begin{aligned}
						\beta_{k+1} &= \left\{
						\begin{aligned}
							&\beta_k + \delta_\beta, \qquad &&\text{if}~\max\{\beta_k, 1/t^{k+1}\} < \|(x^{k+1},y^{k+1}) - (x^{k},y^{k})\|^{-1}, \\
							&\beta_k, \qquad &&\text{otherwies}.
						\end{aligned}\right.\\
						\rho_{k+1} &= \frac{1}{2}\beta_{k+1}L_f + \sigma. \\
					\end{aligned}
				\end{equation*}
				\item[5.] Set $k:=k+1$.
		\end{itemize}}
		\EndFor
	\end{algorithmic}
\end{algorithm}

The convergence of iPL-DCA follows from Theorem \ref{convergence_thm} and Proposition \ref{lin_Prop_bounded_beta} directly.

\begin{theorem}
	Assume that $F$ is bounded below on $C$, $f$ is $L_f$ smooth on $C$.
	Let the sequence $\{(x^k,y^k)\}$ be generated by Algorithm \ref{lin_DCA2}. Moreover assume that  the partial subgradient formula (\ref{partialdnew1}) holds at every iterate point $(x^k,y^k)$ and  $KT(x^k,y) \not =\emptyset$ for all $y\in S(x^k)$. Suppose that either $\epsilon >0$ or $\epsilon = 0$ and the penalty parameter sequence $\{\beta_k\}$ is bounded. Then any accumulation point of $\{(x^k,y^k)\}$
	is a KKT point of problem $({\rm VP})_{\epsilon}$.
\end{theorem}

\section{Numerical experiments on SV bilevel model selection}\label{sec_SV}
In this section, we will conduct numerical experiments on the SV bilevel model selection  problem (SVBP). 
Extensive numerical experiments of iP-DCA on multiple hyperparameter selection models (e.g., elastic net, sparse group lasso, low-rank matrix completion and etc) are presented in \cite{icml2022}.

Let $x := (\mu,\bar{\mathbf{w}}) \in \mathbb{R}^{n+1}$, $y :=( \mathbf{w}^1,\dots, \mathbf{w}^T, \mathbf{c}) \in \mathbb{R}^{(n+1)T}$, $X = [\frac{1}{ \lambda_{ub}}, \frac{1}{\lambda_{lb}}] \times [\bar{\mathbf{w}}_{lb}, \bar{\mathbf{w}}_{ub}]$, $Y = \mathbb{R}^{(n+1)T}$,
\begin{equation*}
	\begin{aligned}
		f(x,y) 
		= \sum_{t=1}^{T} \left(\frac{\|\mathbf{w}^t\|^2}{2\mu} + \sum_{j\in \Omega_{trn}^t}  \max(1 - b_j( \mathbf{a}_j^T\mathbf{w}^t-c_t), 0 ) \right),
	\end{aligned}
\end{equation*} 
and
\begin{equation*}\label{svm_g}
	g(x,y) = \begin{pmatrix}
		g_1(x,y) \\ \vdots \\ g_T(x,y)
	\end{pmatrix} \quad \text{with } \quad	g_t(x,y) = \begin{pmatrix}
		-\bar{\mathbf{w}} -  \mathbf{w}^t \\
		\mathbf{w}^t - \bar{\mathbf{w}} 
	\end{pmatrix}, t= 1, \ldots, T.
\end{equation*}
Obviously  $F$, $f$ and $g$ are all convex functions defined an open set containing $X \times Y$,  and problem (SVBP) can be regarded as a special case of the DC bilevel program  (DCBP).  {Both $F$  and $f$ are bounded below on $X\times Y$}.	When $\bar{\mathbf{w}}_{ub} \ge \bar{\mathbf{w}}_{lb} > 0$, $\mathcal{F}(x) \not =\emptyset$ for an open set containing $X$. And since $b_j \in \{-1,1\}$, $f(x,y)$ is coercive and continuous with respect to lower-level variable $y$ for any given $x$ in an open set containing $X$, thus $S(x) \not =\emptyset$ for all $x$ in an open set containing $X$.
The function $g$ is smooth and $f$ is a sum of a smooth function and a function which is independent of variable $x$. Hence by Proposition \ref{partiald}, the partial differential formula  (\ref{partialdnew1}) holds at each point $(x,y)$.
Since the lower level constraints are affine,  KKT conditions holds at any $y \in S(x)$ for any  $x \in X$. 
Therefore, all conditions required by the convergence results of iP-DCA in Theorem 
\ref{convergence2} 
are satisfied.

We now describe how to calculate the main objects that are  required in  iP-DCA on problem (SVBP).  At the
 current iterate $x^k := (\mu^k,\bar{\mathbf{w}}^k)$,  solve $(P_{\mu^k, \bar{\mathbf{w}}^k})$,  the lower level problem parameterized  by $\mu^k,\bar{\mathbf{w}}^k$ and obtain 
 a solution $\tilde{y}^k := (\tilde{\mathbf{w}}^1,\dots, \tilde{\mathbf{w}}^T, \tilde{\mathbf{c}}) \in {S}(x^k)$ and a corresponding KKT multiplier $$(\gamma^k_{1,1},\ldots, \gamma^k_{1,T}, \gamma^k_{2,1}, \ldots, \gamma^k_{2,T})  \in KT(x^k,\tilde{y}^k),$$ where $\gamma^k_{1,t}$ and $\gamma^k_{2,t}$ are multipliers corresponding to constraints $-\bar{\mathbf{w}}^k -  \mathbf{w}^t \le 0, t= 1, \ldots, T$ and $\mathbf{w}^t - \bar{\mathbf{w}}^k \le 0, t= 1, \ldots, T$, respectively.  Since $F(x,y)$ is convex, we have $\xi_0^k=0$.  Since $f(x,y)$ and $g(x,y)$ are both smooth in variable $x$,  $\xi_1^k$ can be calculated by
	\begin{equation*}
	\xi_1^k =\begin{pmatrix}
		- \frac{\sum_{t=1}^{T} \|\tilde{\mathbf{w}}^t\|^2}{2(\mu^k)^2} \\
		- \sum_{t=1}^{T} \gamma^k_{1,t} - \sum_{t=1}^{T} \gamma^k_{2,t}
	\end{pmatrix}= \nabla_x f(x^k,\tilde{y}^k) +\sum_{t=1}^T \nabla_x g_t(x^k,\tilde{y}^k)^T  \gamma_t^k\subseteq  \partial v(x^k),
	\end{equation*} 
	where $\gamma_i^k:=(\gamma_{1,t}^k,\gamma_{2,t}^k)$.  With these objects calculated, we can then carry out the rest of steps in Algorithm \ref{DCA2}.

 Although the SV bilevel model selection  is a nonsmooth bilevel program,  by using auxiliary variables, the problem can be reformulated as a smooth bilevel program with a convex lower level program for which MPEC approach and the iPL-DCA are both applicable. This approach has been taken in  \cite{kunapuli2008classification}  in which  some nonlinear program solver has been used to solve the resulting MPEC.   
Because the smooth lower level objective in the reformulated bilevel program consists of $\sum_{t=1}^{T} \frac{\|\mathbf{w}^t\|^2}{2\mu}$, and the Lipschitz constant $L_f$ of the gradient of $\sum_{t=1}^{T} \frac{\|\mathbf{w}^t\|^2}{2\mu}$ with respect to variables $( \mathbf{w}^1,\dots, \mathbf{w}^T)$ and $\mu$ will be extremely large when $\mu$ is optimized over the interval with small values. According to the update rule for the regularizer parameter $\rho_k$ in the iPL-DCA, $\rho_k$, as the coefficient of the regularizer terms in the subproblem during each iteration, is linear w.r.t. $L_f$ and will be extremely large. For this reason, the iPL-DCA is not a good choice for this problem.
 In next subsection, we will compare our proposed algorithms with the MPEC approach considered in \cite{kunapuli2008classification}.
In numerical experiments, we will follow the suggestions given in \cite{kunapuli2008classification} to replace the complementarity constraints with the relaxed complementarity constraints.
As claimed by \cite{kunapuli2008classification}, such approach can facilitate an early termination of cross-validation and ease the difficulty of dealing the complementarity constraints for nonlinear program solver.
\subsection{Numerical tests}\label{numerical1}
All the numerical experiments are implemented on a laptop with Intel(R) Core(TM) i7-9750H CPU@ 2.60GHz and 32.00 GB memory. All the codes are implemented on MATLAB 2019b. The subproblems in iP-DCA are all convex optimization problems and we apply the Matlab software package SDPT3 \cite{SDPT3,SDPT32003} with default setting to solve them. MPEC problem 
is solved by using $\mathbf{fmincon}$ in Matlab optimization toolbox with setting $'\mathbf{Algorithm}'$ being $'\mathbf{interior-point}'$, $'\mathbf{MaxIterations}'$ being $200$ and $'\mathbf{MaxFunctionEvaluations}'$ being $10^6$. MPEC approach is implemented with low and strict tolerance by setting $'\mathbf{OptimalityTolerance}'$, $'\mathbf{ConstraintTolerance}'$, $'\mathbf{StepTolerance}'$ being $tol = 10^{-2}$, and $10^{-6}$. As $\mathbf{fmincon}$ needs extremely long time to solve large dimension MPEC problems, we first use small size datasets to conduct the numerical comparison between iP-DCA and MPEC approach. We test here three real datasets ``australian\_scale", ``breast-cancer\_scale" and ``diabetes\_scale" downloaded from the SVMLib repository  \cite{SVMLib}\footnote{ http://www.csie.ntu.edu.tw/~cjlin/libsvmtools/datasets/.}. 
Each dataset  is randomly split into a training set $\Omega$ with $|{\Omega}|=\ell_{train}$ data pairs, which is used in the cross-validation bilevel model and a hold-out test set  ${\cal N}$ with $|{\cal N}|=\ell_{test}$ data pairs. We give the descriptions of datasets in Table \ref{dataset}. For each dataset, we use a three-fold cross-validation in the SV bilevel model  selection problem  (SVBP), i.e. $T = 3$, and that each training fold consists of two-thirds of the total training data and validation fold consists of one-third of the total training data. We repeat the experiments 20 times for each dataset.
The values of parameters in SV bilevel model selection (SVBP) are set as: $\lambda_{lb} = 10^{-4}$, $\lambda_{ub} = 10^4$, $\bar{\mathbf{w}}_{lb} = 10^{-6}$ and $\bar{\mathbf{w}}_{ub} = 1.5$. 
For our approach, we test three different values of relaxation parameter $\epsilon \in \{0, 10^{-2}, 10^{-4}\}$ in $({\rm VP})_{\epsilon}$. And the value of relaxation parameter of the relaxed complementarity constraints in MPEC is set to be 
$10^{-6}$. The initial points for both iP-DCA and MPEC approach are chosen as $\lambda = 1$, $\bar{\mathbf{w}} = 0.1\mathbf{e}$, where $\mathbf{e}$ denotes the vector whose elements are all equal to $1$, and the values of other variables are all equal to $0$. These settings are used for all experiments. Parameters in iP-DCA are set as $\beta_0 = 1$, $\rho = 10^{-2}$ and $\delta_\beta = 5$. And we terminate iP-DCA when $t^{k+1} < 10^{-4}$ and $\|(x^{k+1},y^{k+1}) - (x^{k},y^{k})\|/(1+\|(x^{k},y^{k})\|) < tol$.

For each experiment, after we obtain the  hyperparameters  $\hat{\mu}$ and $\hat{\bar{\mathbf{w}}}$ from  implementing our proposed algorithm  and MPEC approach {for the  SV bilevel model selection}, 
we calculate their corresponding cross-validation error (CV error) and test error for comparing the performances of these two methods. 
For calculating the CV error, we put $\hat{\mu}$ and $\hat{\bar{\mathbf{w}}}$ back to the lower level problem in problem (SVBP)  to get the corresponding lower level solution $(\hat{\mathbf{w}}^1,\dots, \hat{\mathbf{w}}^T,\hat{\mathbf{c}})$  and calculate the corresponding cross-validation error $\Theta(\hat{\mathbf{w}}^1,\dots, \hat{\mathbf{w}}^T,\hat{\mathbf{c}})$.
Next, as in \cite{kunapuli2008classification} we implement a post-processing procedure to calculate the generalization error on the hold-out data for each instance. 
In particular as suggested in \cite{kunapuli2008classification}, since only two thirds of the data in $\Omega$ was used in each fold while  in testing we use all the training data from $\Omega$, we should solve the following support vector classification problem with  $\frac{3}{2}\hat{\lambda}=\frac{3}{2\hat{\mu}}$ and $\hat{\bar{\mathbf{w}}}$ as hyperparameter
$$ \underset{\tiny \begin{matrix}
				-\hat{\bar{\mathbf{w}}} \le \mathbf{w} \le \hat{\bar{\mathbf{w}}} \\ c \in \mathbb{R}	\end{matrix} }{\mathrm{min}} \left\{ \frac{3}{4\hat{\mu}}\|\mathbf{w}\|^2 + \sum_{j\in  \Omega}\max( 1 -b_j( \mathbf{a}_j^T\mathbf{w}-c),0) \right\}$$  to obtain the final classifier $(\hat{\mathbf{w}}, \hat{c})$. Then  the test (hold-out) error rate is calculated as:
\[
\hbox{Test error} = \frac{1}{\ell_{test}} \sum_{i\in {\cal N} } \frac{1}{2} |  \mathrm{sign}(\mathbf{a}_i^T\hat{\mathbf{w}}-\hat{c})-b_i |,
\]
where sign($x$) denote the sign function. Note that for each $(\mathbf{a}_i, b_i)$ in the test set ${\cal N}$,   $|\mathrm{sign}(\mathbf{a}_i^T\hat{\mathbf{w}}-\hat{c})-b_i|$ is either equal to zero or $2$ and hence the test error is the average misclassification by the final classifier.
The achieved numerical results averaged over 20 repetitions for each dataset are reported in Table \ref{numerical_result}. 

We compare the computational performance of iP-DCA with different values of $\epsilon$ and $tol$, i.e., $\epsilon = 0, 10^{-4}, 10^{-2}$ and $tol = 10^{-2}, 10^{-3}$ and the MPEC approach with different values of $tol$, i.e., $tol = 10^{-2}, 10^{-6}$. 
Observe from
Table \ref{numerical_result} that different values of $\epsilon$ and $tol$ do not influence the cross-validation error and test error obtained by iP-DCA a lot. The case $\epsilon = 0$ takes more time to achieve desired tolerance on some data sets.
This may be because that, as we have shown in Propositions \ref{Prop4.3} and \ref{prop4.2}, when $\epsilon = 0$, NNAMCQ never hold for the problem (VP), the ENNAMCQ always holds for the problem $({\rm VP})_{\epsilon}$ when $\epsilon > 0$. As a consequence, the problem (VP) is more ill-conditioned  compared to the problem $({\rm VP})_{\epsilon}$ with $\epsilon > 0$, then iP-DCA performs better on the problem $({\rm VP})_{\epsilon}$ with $\epsilon > 0$. Compared with MPEC approach, our proposed iP-DCA achieves a smaller cross-validation error, which is exactly the value of upper level objective of the bilevel problem (SVBP), on datasets ``breast-cancer\_scale'' and ``diabetes\_scale".  Furthermore, the time taken by our proposed iP-DCA is shorter than MPEC approach. The test errors of our proposed iP-DCA and MPEC approach are similar and iP-DCA achieves a smaller test error than MPEC approach on dataset ``diabetes\_scale". {Moreover} both iP-DCA and MPEC approach can obtain a relatively good solution without requiring a small $tol$.


Next, we are going to test our proposed iP-DCA on two large scale datasets ``mushrooms" and ``phishing" downloaded from the SVMLib repository. The descriptions of datasets are given in Table \ref{dataset}. We set $tol = 10^{-2}$ for these tests. The numerical results averaged over 20 repetitions for each data set are reported in Table \ref{numerical_result2}. It can be observed from Table \ref{numerical_result2} that different values of $\epsilon$ and $tol$ do not influence the cross-validation error obtained by iP-DCA much but the case $\epsilon = 0$ takes more time to achieve desired tolerance on some data sets. And iP-DCA can obtain a satisfactory solution within an acceptable time on large scale problems.

\begin{table}[!htbp]
	\centering
	\caption{Description of datasets used} \label{dataset}
	\begin{tabular}{p{120pt} p{30pt} p{30pt} p{30pt} c}
		\toprule
		Dataset & $\ell_{train}$ & $\ell_{test}$ & $n$ & T \\
		\midrule
		australian\_scale & 345 & 345 & 14 & 3 \\
		breast-cancer\_scale & 339 & 344 & 10 &3 \\
		diabetes\_scale & 384 & 384 & 8 & 3 \\
		mushrooms & 4062 & 4062 & 112 & 3 \\
		phishing & 5526 & 5529 & 68 &3 \\
		\bottomrule
	\end{tabular}
\end{table}

\begin{table}[!htbp]
	\centering
	\caption{Numerical results comparing iP-DCA and MPEC approach} \label{numerical_result}
	\resizebox{\textwidth}{!}{
		\begin{tabular}{cp{140pt} p{60pt} p{60pt} c}
			\toprule
			Dataset & Method & CV error & Test error & Time(sec) \\
			\midrule 
			\multirow{7}{*}{australian\_scale}
			& iP-DCA($\epsilon = 0$, $tol = 10^{-2}$) & 0.28 $\pm$ 0.03 & 0.15 $\pm$ 0.01 & 70.0 $\pm$ 116.9 \\ 
			& iP-DCA($\epsilon = 0$, $tol = 10^{-3}$) & 0.28 $\pm$ 0.03 & 0.15 $\pm$ 0.01 & 75.7 $\pm$ 123.8 \\ 
			& iP-DCA($\epsilon = 10^{-2}$, $tol = 10^{-2}$) & 0.28 $\pm$ 0.03 & 0.15 $\pm$ 0.01 & 10.1 $\pm$ 5.1 \\ 
			& iP-DCA($\epsilon = 10^{-2}$, $tol = 10^{-3}$) & 0.28 $\pm$ 0.03 & 0.15 $\pm$ 0.01 & 119.8 $\pm$ 55.8 \\ 
			& iP-DCA($\epsilon = 10^{-4}$, $tol = 10^{-2}$) & 0.28 $\pm$ 0.03 & 0.15 $\pm$ 0.01 & 58.5 $\pm$ 108.8 \\ 
			& iP-DCA($\epsilon = 10^{-4}$, $tol = 10^{-3}$) & 0.28 $\pm$ 0.03 & 0.15 $\pm$ 0.01 & 118.6 $\pm$ 123.6 \\ 
			& MPEC approach($tol = 10^{-2}$) & 0.28 $\pm$ 0.03 & 0.15 $\pm$ 0.01 & 130.0 $\pm$ 82.4 \\  
			& MPEC approach($tol = 10^{-6}$) & 0.28 $\pm$ 0.03 & 0.15 $\pm$ 0.01 & 391.3 $\pm$ 226.9 \\ 
			\midrule
			\multirow{7}{*}{breast-cancer\_scale}
			& iP-DCA($\epsilon = 0$, $tol = 10^{-2}$) & 0.07 $\pm$ 0.01 & 0.03 $\pm$ 0.01 & 26.2 $\pm$ 19.5 \\ 
			& iP-DCA($\epsilon = 0$, $tol = 10^{-3}$) & 0.06 $\pm$ 0.01 & 0.03 $\pm$ 0.01 & 69.6 $\pm$ 52.3 \\ 
			& iP-DCA($\epsilon = 10^{-2}$, $tol = 10^{-2}$) & 0.07 $\pm$ 0.01 & 0.03 $\pm$ 0.01 & 15.6 $\pm$ 4.5 \\ 
			& iP-DCA($\epsilon = 10^{-2}$, $tol = 10^{-3}$) & 0.07 $\pm$ 0.01 & 0.03 $\pm$ 0.01 & 106.5 $\pm$ 44.8 \\ 
			& iP-DCA($\epsilon = 10^{-4}$, $tol = 10^{-2}$) & 0.07 $\pm$ 0.01 & 0.03 $\pm$ 0.01 & 17.5 $\pm$ 6.3 \\ 
			& iP-DCA($\epsilon = 10^{-4}$, $tol = 10^{-3}$) & 0.06 $\pm$ 0.01 & 0.03 $\pm$ 0.01 & 72.5 $\pm$ 55.0 \\ 
			& MPEC approach($tol = 10^{-2}$) & 0.09 $\pm$ 0.01 & 0.03 $\pm$ 0.01 & 54.3 $\pm$ 13.7 \\  
			& MPEC approach($tol = 10^{-6}$) & 0.09 $\pm$ 0.01 & 0.03 $\pm$ 0.01 & 226.0 $\pm$ 108.4 \\    
			\midrule
			\multirow{7}{*}{diabetes\_scale}
			& iP-DCA($\epsilon = 0$, $tol = 10^{-2}$) & 0.55 $\pm$ 0.02 & 0.24 $\pm$ 0.02 & 14.5 $\pm$ 30.3 \\ 
			& iP-DCA($\epsilon = 0$, $tol = 10^{-3}$) & 0.55 $\pm$ 0.02 & 0.24 $\pm$ 0.02 & 24.2 $\pm$ 44.8 \\ 
			& iP-DCA($\epsilon = 10^{-2}$, $tol = 10^{-2}$) & 0.56 $\pm$ 0.02 & 0.24 $\pm$ 0.02 & 3.1 $\pm$ 0.4 \\ 
			& iP-DCA($\epsilon = 10^{-2}$, $tol = 10^{-3}$) & 0.56 $\pm$ 0.02 & 0.24 $\pm$ 0.02 & 65.7 $\pm$ 33.6 \\ 
			& iP-DCA($\epsilon = 10^{-4}$, $tol = 10^{-2}$) & 0.55 $\pm$ 0.02 & 0.24 $\pm$ 0.02 & 18.2 $\pm$ 27.1 \\ 
			& iP-DCA($\epsilon = 10^{-4}$, $tol = 10^{-3}$) & 0.55 $\pm$ 0.02 & 0.24 $\pm$ 0.02 & 45.1 $\pm$ 53.8 \\ 
			& MPEC approach($tol = 10^{-2}$) & 0.59 $\pm$ 0.02 & 0.26 $\pm$ 0.02 & 52.1 $\pm$ 42.2 \\  
			& MPEC approach($tol = 10^{-6}$) & 0.59 $\pm$ 0.02 & 0.26 $\pm$ 0.02 & 326.8 $\pm$ 312.0 \\   
			\bottomrule
		\end{tabular}
	}
\end{table}

\begin{table}[!htbp]
	\centering
	\caption{Numerical results of iP-DCA on datasets ``mushrooms" and ``phishing" with $tol = 10^{-2}$} \label{numerical_result2}
	\resizebox{\textwidth}{!}{
		\begin{tabular}{cp{100pt} p{100pt} p{100pt} c}
			\toprule
			Dataset & Method & CV error & Test error & Time(sec) \\
			\midrule 
			\multirow{3}{*}{mushrooms} & iP-DCA($\epsilon = 0$) & 6.36e-04 $\pm$ 5.94e-04 & 0 $\pm$ 0 & 334.3 $\pm$ 346.1 \\ 
			& iP-DCA($\epsilon = 10^{-2}$) & 1.53e-03 $\pm$ 3.85e-03 & 3.57e-04 $\pm$ 1.34e-03 & 109.3 $\pm$ 35.2 \\ 
			& iP-DCA($\epsilon = 10^{-4}$) & 6.38e-04 $\pm$ 6.08e-04 & 0 $\pm$ 0 & 162.9 $\pm$ 27.4 \\ 
			\midrule
			\multirow{3}{*}{phishing} & iP-DCA($\epsilon = 0$) & 0.29 $\pm$ 0.00 & 0.09 $\pm$ 0.00 & 357.9 $\pm$ 95.2 \\ 
			& iP-DCA($\epsilon = 10^{-2}$) & 0.29 $\pm$ 0.00 & 0.09 $\pm$ 0.00 & 222.1 $\pm$ 18.9 \\ 
			& iP-DCA($\epsilon = 10^{-4}$) & 0.29 $\pm$ 0.00 & 0.09 $\pm$ 0.00 & 215.4 $\pm$ 46.5 \\ 
			\bottomrule
	\end{tabular}}
\end{table}

{
\subsection{Further numerical tests}
In the follow-up paper \cite{icml2022}, experiments on the SV bilevel model selection  problem (SVBP) and comparisions with 
the state-of-the-art approaches in machine learning community, including the grid search, the random search and the tree-structured Parzen estimator Bayesian approach (TPE) \cite{bergstra2013making} have been conducted. In this subsection we summarize these results.

All the numerical experiments are implemented on a computer with Intel(R) Core(TM) i9-9900K CPU @ 3.60GHz and 16.00 GB memory. All the codes are implemented 
in Python and are available at \url{https://github.com/SUSTech-Optimization/VF-iDCA}.
Six real datasets ``liver-disorders\_scale'', ``diabetes\_scale'', ``breast-cancer\_scale'', ``sonar'', ``a1a'', ``w1a'' collected from the SVMLib repository are tested. For each dataset,  the SV bilevel model selection problem (SVBP) with three-fold and six-fold cross-validations, i.e. $T = 3, 6$ are solved respectively. Each dataset is randomly split in the same way as in Section \ref{numerical1}. The experiments are repeated 30 times for each dataset.
The values of parameters in SV bilevel model selection (SVBP) are set as: $\lambda_{lb} = 10^{-4}$, $\lambda_{ub} = 10^4$, $\bar{\mathbf{w}}_{lb} = 10^{-6}$ and $\bar{\mathbf{w}}_{ub} = 10$.  

For the implementation of iP-DCA, unlike in Section \ref{numerical1}, in which iP-DCA is applied to solve problem (SVBP) directly, for the numerical experiments in this part, the hyperparameter decoupling technique is applied to reformulate the problem (SVBP) into a  DC bilevel program  (DCBP) (see \cite{icml2022} for details) before applying iP-DCA. The strongly convex subproblem in iP-DCA at each iteration is solved by using the CVXPY package. 
Parameters in iP-DCA are set as $\epsilon = 0$, $\beta_0 = 1$ and $\delta_\beta = 5$. And  iP-DCA is terminated when $\max \{\|(x^{k+1},y^{k+1}) - (x^{k},y^{k})\|/(1+\|(x^{k},y^{k})\|),t^{k+1}\} < tol$.
The computational performance of iP-DCA is compared using two different values of $tol$, i.e., $tol = 10^{-1}, 10^{-2}$.

For the implementation of the grid search and the random search, the searches are run over two-dimension hyperparameter $\theta = (\theta_1, \theta_2)$ on \\ $\{-4,-3, \ldots, 3, 4\} \times \{ -6, -5, \ldots, 1, 2\}$ and with setting $\mu = 10^{\theta_1}$ and $\bar{\mathbf{w}} = (10^{\theta_2}, \dots, 10^{\theta_2})^\top$ in the problem (SVBP).
The subproblems in the grid search and the random search are all solved by using the CVXPY package.
And for the implementation of  TPE,  the hyperparameter $\log_{10}(\mu)$ in $[-4, 4]$, and the hyperparameter $\log_{10}(\bar{w}_i)$ in $[-6, 2]$ are searched, respectively. Because  TPE will be extremely slow when the dimension of the hyperparameters is too high,  the maximum number of iteration of  TBE is set to be $10$. And the TPE method is also tested on the simplified model with a two-dimension hyperparameter, that has  the same setting as the one solved by the search methods. This method is denoted by ``TPE2''. For this simplified ``TPE2'',  the maximum number of iterations of the TPE method are set to be $100$. 
The TPE method is implemented using the code collected from \url{https://github.com/hyperopt/hyperopt} and its subproblem is solved by using the CVXPY package. In the implementation of all the methods, the CVXPY package is set with using the open source solvers ECOS and SCS. 

The achieved numerical results of the  three-fold and the six-fold SV bilevel model selection problem averaged over 30 repetitions for each dataset are reported in Tables \ref{table:svm_full_3} and \ref{table:svm_full_6}, respectively. The cross-validation error (CV error) and the test error are calculated in the same way as in Section \ref{numerical1}.
Observe from
Table \ref{table:svm_full_3} and \ref{table:svm_full_6} that compared with the state of the art approaches,
including the grid search method, the random search method and the TPE, our proposed iP-DCA shows superiority by achieving a smaller cross-validation error and also a smaller test error. Furthermore, the time spent by our proposed iP-DCA is shorter than other approaches on most of the datasets.
It can be also observed that on all the datasets except ``breast-cancer\_scale'',
different values of $tol$ do not influence the cross-validation error and test error obtained by iP-DCA a lot. In view of the test error, iP-DCA can always obtain a relatively good solution without requiring a tight tolerance. This suggests to set a moderate algorithmic tolerance for iP-DCA when we apply it on practical problems for obtaining a satisfactory solution with shorter computation time.

\begin{table}[!htbp]
	\centering
	\caption{Numerical results of three-fold SV bilevel model selection problem on datasets ``liver-disorders\_scale'', ``diabetes\_scale'', ``breast-cancer\_scale'', ``sonar'', ``a1a'', ``w1a''} \label{table:svm_full_3}
	\resizebox{\textwidth}{!}{
		\begin{tabular}{cp{100pt} p{100pt} p{100pt} c}
			\toprule
			Dataset & Method & CV error & Test error & Time(sec) \\
			\midrule 
			\multirowcell{6}{liver-disorders\_scale \\ $\ell_{train} = 72$ \\ $\ell_{test} = 73$ \\ $n = 5$} & iP-DCA($tol = 10^{-1}$) & 0.53 $\pm$ 0.07 & 0.27 $\pm$ 0.03  & 0.09 $\pm$ 0.02 \\ 
			&iP-DCA($tol = 10^{-2}$)  & 0.53 $\pm$ 0.09 & 0.28 $\pm$ 0.05 & 0.20 $\pm$ 0.06 \\ 
			& Grid Search & 0.64 $\pm$ 0.08 & 0.34 $\pm$ 0.06 & 0.53 $\pm$ 0.01 \\ 
			& Random Search & 0.58 $\pm$ 0.06 & 0.32 $\pm$ 0.05 & 0.56 $\pm$ 0.03 \\ 
			& TPE & 0.65 $\pm$ 0.07 & 0.34 $\pm$ 0.05 & 0.37 $\pm$ 0.29 \\ 
			& TPE2 &  0.61 $\pm$ 0.07 & 0.33 $\pm$ 0.06 & 2.88 $\pm$ 1.16 \\ 
			\midrule
			\multirowcell{6}{diabetes\_scale \\ $\ell_{train} = 384$ \\ $\ell_{test} = 384$ \\ $n = 8$} & iP-DCA($tol = 10^{-1}$) & 0.48 $\pm$ 0.02 & 0.23 $\pm$ 0.01 & 0.18 $\pm$ 0.02 \\ 
			&iP-DCA($tol = 10^{-2}$)  & 0.48 $\pm$ 0.02 & 0.23 $\pm$ 0.01 & 0.28 $\pm$ 0.03 \\ 
			& Grid Search & 0.55 $\pm$ 0.03 & 0.33 $\pm$ 0.05 & 1.70 $\pm$ 0.11\\ 
			& Random Search & 0.56 $\pm$ 0.04 & 0.30 $\pm$ 0.06 & 1.83 $\pm$ 0.09 \\ 
			& TPE & 0.55 $\pm$ 0.03 & 0.29 $\pm$ 0.05 & 6.64 $\pm$ 4.30 \\ 
			& TPE2 & 0.54 $\pm$ 0.03 & 0.32 $\pm$ 0.06 & 18.67 $\pm$ 7.84 \\ 
			\midrule
			\multirowcell{6}{breast-cancer\_scale \\ $\ell_{train} = 388$ \\ $\ell_{test} = 345$ \\ $n = 10$} & iP-DCA($tol = 10^{-1}$) & 0.09 $\pm$ 0.01 & 0.04 $\pm$ 0.01 & 0.14 $\pm$ 0.01  \\ 
			&iP-DCA($tol = 10^{-2}$)  & 0.05 $\pm$ 0.01 & 0.03 $\pm$ 0.01 &  1.12 $\pm$ 0.59 \\ 
			& Grid Search &  0.08 $\pm$ 0.01 &  0.12 $\pm$ 0.06 & 1.63 $\pm$ 0.04 \\ 
			& Random Search & 0.09 $\pm$ 0.01 & 0.08 $\pm$ 0.09 & 1.80 $\pm$ 0.03 \\ 
			& TPE &  0.09 $\pm$ 0.01 &  0.10 $\pm$ 0.11 &  9.14 $\pm$ 4.55 \\ 
			& TPE2 &  0.07 $\pm$ 0.01 &  0.09 $\pm$ 0.10 & 14.72 $\pm$ 6.02 \\ 
						\midrule
			\multirowcell{6}{sonar \\ $\ell_{train} = 102$ \\ $\ell_{test} = 106$ \\ $n = 60$} & iP-DCA($tol = 10^{-1}$) & 0.03 $\pm$ 0.02 & 0.24 $\pm$ 0.04 & 0.48 $\pm$ 0.09 \\ 
			&iP-DCA($tol = 10^{-2}$)  &  0.00 $\pm$ 0.00 & 0.24 $\pm$ 0.04 & 2.22 $\pm$ 1.50 \\ 
			& Grid Search & 0.58 $\pm$ 0.08 & 0.40 $\pm$ 0.12 & 3.19 $\pm$ 0.10 \\ 
			& Random Search &  0.54 $\pm$ 0.06 & 0.34 $\pm$ 0.10 &  3.23 $\pm$ 0.06 \\ 
			& TPE &  0.64 $\pm$ 0.10 & 0.41 $\pm$ 0.12 & 40.77 $\pm$ 7.12 \\ 
			& TPE2 &  0.57 $\pm$ 0.08 &  0.37 $\pm$ 0.13 & 18.47 $\pm$ 6.84\\ 
						\midrule
			\multirowcell{6}{a1a \\ $\ell_{train} = 801$ \\ $\ell_{test} = 804$ \\ $n = 123$} & iP-DCA($tol = 10^{-1}$) & 0.27 $\pm$ 0.02  & 0.17 $\pm$ 0.01 & 1.10 $\pm$ 0.07 \\ 
			&iP-DCA($tol = 10^{-2}$)  &  0.27 $\pm$ 0.02 &  0.18 $\pm$ 0.01 & 10.17 $\pm$ 5.47 \\ 
			& Grid Search & 0.41 $\pm$ 0.02 & 0.24 $\pm$ 0.02 & 8.04 $\pm$ 0.15 \\ 
			& Random Search & 0.41 $\pm$ 0.02 & 0.22 $\pm$ 0.03 & 8.62 $\pm$ 0.30 \\ 
			& TPE &  0.42 $\pm$ 0.03 & 0.23 $\pm$ 0.03 &  176.59 $\pm$ 17.38 \\ 
			& TPE2 &  0.41 $\pm$ 0.02 & 0.24 $\pm$ 0.02 &  65.51 $\pm$ 16.24 \\ 
						\midrule
			\multirowcell{6}{w1a\\ $\ell_{train} = 1236$ \\ $\ell_{test} = 1241$ \\ $n = 300$} & iP-DCA($tol = 10^{-1}$) & 0.01 $\pm$ 0.00  & 0.02 $\pm$ 0.00 & 4.87 $\pm$ 0.51 \\ 
			&iP-DCA($tol = 10^{-2}$)  & 0.01 $\pm$ 0.00 & 0.02 $\pm$ 0.00 & 27.49 $\pm$ 7.31 \\ 
			& Grid Search & 0.06 $\pm$ 0.01 &  0.03 $\pm$ 0.00 & 20.21 $\pm$ 0.82 \\ 
			& Random Search &  0.06 $\pm$ 0.01 & 0.03 $\pm$ 0.00 &  20.44 $\pm$ 1.10 \\ 
			& TPE &  0.06 $\pm$ 0.01 &  0.03 $\pm$ 0.00 & 299.62 $\pm$ 78.72 \\ 
			& TPE2 & 0.06 $\pm$ 0.01 &  0.03 $\pm$ 0.00 & 86.10 $\pm$ 28.19 \\ 
			\bottomrule
	\end{tabular}}
\end{table}

\begin{table}[!htbp]
	\centering
	\caption{Numerical results of six-fold SV bilevel model selection problem on datasets ``liver-disorders\_scale'', ``diabetes\_scale'', ``breast-cancer\_scale'', ``sonar'', ``a1a'', ``w1a''} \label{table:svm_full_6}
	\resizebox{\textwidth}{!}{
		\begin{tabular}{cp{100pt} p{100pt} p{100pt} c}
			\toprule
			Dataset & Method & CV error & Test error & Time(sec) \\
			\midrule 
			\multirowcell{6}{liver-disorders\_scale \\ $\ell_{train} = 72$ \\ $\ell_{test} = 73$ \\ $n = 5$} & iP-DCA($tol = 10^{-1}$) & 0.41 $\pm$ 0.08 & 0.27 $\pm$ 0.04  & 0.18 $\pm$ 0.03 \\ 
			&iP-DCA($tol = 10^{-2}$)  &  0.41 $\pm$ 0.08 & 0.27 $\pm$ 0.05 &  0.33 $\pm$ 0.14 \\ 
			& Grid Search & 0.63 $\pm$ 0.08 & 0.33 $\pm$ 0.07 & 0.78 $\pm$ 0.02 \\ 
			& Random Search & 0.62 $\pm$ 0.07 &  0.31 $\pm$ 0.05 & 0.79 $\pm$ 0.04 \\ 
			& TPE & 0.63 $\pm$ 0.08 & 0.34 $\pm$ 0.05 &  1.06 $\pm$ 1.04 \\ 
			& TPE2 &  0.62 $\pm$ 0.07 & 0.32 $\pm$ 0.06 &  6.88 $\pm$ 4.15 \\ 
			\midrule
			\multirowcell{6}{diabetes\_scale \\ $\ell_{train} = 384$ \\ $\ell_{test} = 384$ \\ $n = 8$} & iP-DCA($tol = 10^{-1}$) & 0.43 $\pm$ 0.02 & 0.23 $\pm$ 0.01 & 0.35 $\pm$ 0.01 \\ 
			&iP-DCA($tol = 10^{-2}$)  &  0.43 $\pm$ 0.02 & 0.23 $\pm$ 0.01 & 0.56 $\pm$ 0.08 \\ 
			& Grid Search &  0.55 $\pm$ 0.03 & 0.32 $\pm$ 0.05 & 3.18 $\pm$ 0.14 \\ 
			& Random Search & 0.56 $\pm$ 0.03 & 0.31 $\pm$ 0.05 & 3.63 $\pm$ 0.21 \\ 
			& TPE & 0.55 $\pm$ 0.03 & 0.27 $\pm$ 0.06 &  29.52 $\pm$ 13.13 \\ 
			& TPE2 & 0.55 $\pm$ 0.03 & 0.33 $\pm$ 0.05 &  51.85 $\pm$ 20.49 \\ 
			\midrule
			\multirowcell{6}{breast-cancer\_scale \\ $\ell_{train} = 388$ \\ $\ell_{test} = 345$ \\ $n = 10$} & iP-DCA($tol = 10^{-1}$) & 0.08 $\pm$ 0.01 & 0.04 $\pm$ 0.01 & 0.29 $\pm$ 0.08   \\ 
			&iP-DCA($tol = 10^{-2}$)  & 0.03 $\pm$ 0.01 & 0.03 $\pm$ 0.01 &   2.01 $\pm$ 0.17 \\ 
			& Grid Search &  0.08 $\pm$ 0.02 &  0.15 $\pm$ 0.06 &  3.38 $\pm$ 0.25 \\ 
			& Random Search & 0.09 $\pm$ 0.02 & 0.07 $\pm$ 0.08 &  3.92 $\pm$ 0.29 \\ 
			& TPE &  0.09 $\pm$ 0.01 &   0.11 $\pm$ 0.13 &  25.96 $\pm$ 12.95 \\ 
			& TPE2 &  0.07 $\pm$ 0.02 & 0.08 $\pm$ 0.09 &  38.69 $\pm$ 16.27 \\ 
			\midrule
			\multirowcell{6}{sonar \\ $\ell_{train} = 102$ \\ $\ell_{test} = 106$ \\ $n = 60$} & iP-DCA($tol = 10^{-1}$) & 0.00 $\pm$ 0.00 & 0.23 $\pm$ 0.04 & 0.92 $\pm$ 0.02 \\ 
			&iP-DCA($tol = 10^{-2}$)  &  0.00 $\pm$ 0.00 & 0.23 $\pm$ 0.04 & 0.92 $\pm$ 0.02 \\ 
			& Grid Search &  0.59 $\pm$ 0.08 & 0.39 $\pm$ 0.11 & 6.57 $\pm$ 0.32 \\ 
			& Random Search &   0.54 $\pm$ 0.06 & 0.32 $\pm$ 0.08 & 6.44 $\pm$ 0.28 \\ 
			& TPE &   0.60 $\pm$ 0.07 & 0.39 $\pm$ 0.12 & 97.65 $\pm$ 31.37 \\ 
			& TPE2 &   0.57 $\pm$ 0.08 &  0.36 $\pm$ 0.12 & 58.19 $\pm$ 29.60 \\ 
			\midrule
			\multirowcell{6}{a1a \\ $\ell_{train} = 801$ \\ $\ell_{test} = 804$ \\ $n = 123$} & iP-DCA($tol = 10^{-1}$) & 0.19 $\pm$ 0.01  & 0.17 $\pm$ 0.01 & 4.22 $\pm$ 0.37 \\ 
			&iP-DCA($tol = 10^{-2}$)  & 0.18 $\pm$ 0.02  &  0.17 $\pm$ 0.01 &  63.01 $\pm$ 186.14 \\ 
			& Grid Search & 0.40 $\pm$ 0.02 &  0.25 $\pm$ 0.02 & 17.60 $\pm$ 0.36 \\ 
			& Random Search & 0.40 $\pm$ 0.02 &  0.21 $\pm$ 0.03 & 18.59 $\pm$ 0.42 \\ 
			& TPE &  0.41 $\pm$ 0.03 & 0.23 $\pm$ 0.03 &  312.63 $\pm$ 60.60 \\ 
			& TPE2 &  0.40 $\pm$ 0.02 & 0.24 $\pm$ 0.02 &  161.68 $\pm$ 42.67 \\ 
			\midrule
			\multirowcell{6}{w1a\\ $\ell_{train} = 1236$ \\ $\ell_{test} = 1241$ \\ $n = 300$} & iP-DCA($tol = 10^{-1}$) & 0.01 $\pm$ 0.00  & 0.02 $\pm$ 0.00 & 26.74 $\pm$ 3.67 \\ 
			&iP-DCA($tol = 10^{-2}$)  &  0.01 $\pm$ 0.00 & 0.02 $\pm$ 0.00 & 97.50 $\pm$ 35.99 \\ 
			& Grid Search & 0.05 $\pm$ 0.00 & 0.03 $\pm$ 0.00 & 44.29 $\pm$ 1.39 \\ 
			& Random Search &  0.05 $\pm$ 0.00 &  0.03 $\pm$ 0.00 & 61.80 $\pm$ 2.91 \\ 
			& TPE &  0.05 $\pm$ 0.01 & 0.03 $\pm$ 0.00 &  703.72 $\pm$ 82.75 \\ 
			& TPE2 &  0.05 $\pm$ 0.00 & 0.03 $\pm$ 0.00 & 190.04 $\pm$ 39.00 \\ 
			\bottomrule
	\end{tabular}}
\end{table}

}

\section{Concluding remarks}
Motivated by hyperparameter selection problems, in this paper we develop two DCA type algorithms for solving the DC bilevel program. Our numerical experiments on the SV bilevel model selection show that our approach is promising. Due to the space limit, we are not able to present more studies for more complicated models in hyperparameter selection problems. We hope to study these problems in our future work.  Note that all of our results except the result on the constraint qualification ENNAMCQ in Proposition \ref{prop4.2} can be applied to the case where there are also  some extra upper level  constraints  
$G(x,y):=(G_1(x,y), \dots, G_k(x,y)) \leq 0$ as long as each function $G_i(x,y)$ is a difference of convex function.  In this case,  the corresponding approximate bilevel program has an extra DC constraint $G(x,y)\leq 0$.  Although the constraint qualification ENNAMCQ no longer holds automatically,   it is reasonable to impose ENNAMCQ for the corresponding approximate bilevel program to hold.

        {\bf Acknowledgements}. We thank the guest editor and three reviewers for their 
helpful and constructive comments that have helped improve this paper substantially. The alphabetical order of the authors
indicates their equal contributions to the paper.


\vspace*{-0.1in}


\begin{thebibliography}{}
	\bibitem{allende2013solving}
	Allende, G., Still, G.:
	{ Solving bilevel programs with the KKT-approach}.
	{Mathematical  Programming}. \textbf{138}, 309-332 (2013)
	
	\bibitem{KuangYe}
	Bai, K., Ye, J.J.:
	{Directional necessary optimality conditions for bilevel programs.}
	{Mathematics of Operations Research}, \textbf{47}, 1169-1191 (2022)


	\bibitem{bard1998practical}
	Bard, J.:
	Practical Bilevel Optimization: Algorithms and Applications. Kluwer Academic Publishers, Dordrecht (1998)
	
	
	\bibitem{FOMbook} Beck, A.: First-order methods in optimization. Society for Industrial and Applied Mathematics (2017).
	
	\bibitem{BenAyed}
	Ben-Tal, A., Blair, C.:
	{Computational difficulties of bilevel linear programming.}
	{Operations Research}. \textbf{38}, 556-560 (1990)
	
	\bibitem{Bennett} Bennett, K.P., Hu, J., Ji., X., Kunapuli, G., Pang, J.-S.: Model selection via bilevel optimization, in The 2006 IEEE International Joint Conference on Neural Network Proceedings. 1922-1929 (2006)
	
	\bibitem{bergstra2013making} Bergstra, J., Yamins, D., Cox, D.: Making a science of model search: Hyperparameter optimization in hundreds of dimensions for vision architectures, In: International Conference on Machine Learning. \textbf{28}(1), 115-123 (2013)
	
	
	
	
	
	\bibitem{SVMLib}  Chang, C-C., Lin, C-J.: LIBSVM : a library for support vector machines. ACM Transactions on Intelligent Systems and Technology. \textbf{2}(3), 1-27 (2011)
	
	\bibitem{clarke1990optimization} Clarke, F.H.: Optimization and Nonsmooth Analysis. Society for Industrial and Applied Mathematics, Philadelphia (1990)
	
	\bibitem{ClarkeLSW} Clarke, F.H., Ledyaev, Y.S.,  Stern, R.J., Wolenski, P.R.: Nonsmooth Analysis and Control Theorey. Springer Science \& Business Media, New York (1998)
	
	\bibitem{ColsonMarcotteSavard}
	Colson, B., Marcotte, P., Savard, G.:
	An overview of bilevel optimization. Annals of Operations Research. \textbf{153}(1-2), 235-256 (2007)
	
	
	\bibitem{Dempe2002}  Dempe, S.: Foundations of Bilevel Programming. Kluwer Academic Publishers, Dordrecht (2002)
	
	
	\bibitem{Dem-Dut} Dempe, S., Dutta, J.: {Is bilevel programming a special case of mathematical programming with equilibrium constraints?}. Mathematical Programming.  \textbf{131}, 37-48 (2012) 
	
	\bibitem{Dempebook} {Dempe, S., Zemkoho, A.}: { Bilevel Optimization: Advances and Next Challenges}.  Springer Optimization and its Applications, vol. 161 (2020)
	\bibitem{dempe2015bilevel}
	Dempe, S., Kalashnikov, V., Pérez-Valdés, G., Kalashnykova, N.:
	{Bilevel Programming Problems.}
	{Energy Systems},  Springer Science \& Business Media, Berlin (2015)
	
	
	
	
	
	
	
	\bibitem{Franceschi}
	Franceschi, L., Frasconi, P., Salzo, S., Grazzi, R.,  Pontil, M.:
	{Bilevel programming for hyperparameter optimization and meta-learning.}
	In: International Conference on Machine Learning. \textbf{80}, 1568-1577 (2018)

	
	
	\bibitem{icml2022} Gao, L., Ye, J.J., Yin, H., Zeng, S., Zhang, J.: Value function based difference-of-convex algorithm for bilevel hyperparameter selection problems. In: International Conference on Machine Learning. \textbf{162}, 7164-7182 (2022)
	
	
	
	\bibitem{calmness_multifunction} Henrion, R., Jourani, A., Outrata, J.V.: On the calmness of a class of multifunctions. SIAM Journal on Optimization. \textbf{13}, 603–618 (2002)
	
	\bibitem{DC_overview} Horst, R., Thoai, N.V.: DC programming: overview.  Journal of Optimization Theory and Applications. \textbf{103}(1), 1–43 (1999)
	
	\bibitem{Jourani} Jourani, A.: Constraint qualifications and Lagrange multipliers in nondifferentiable programming problems. Journal of Optimization Theory and Applications. \textbf{81}, 533-548 (1994)
	
	\bibitem{Kunapuli}  Kunapuli, G.: A bilevel optimization approach to machine learning. Ph.D Thesis. (2008)
	
	\bibitem{kunapuli2008classification} 
	{Kunapuli, G., Bennett, K.P., Hu, J.,  Pang, J-S.:}
	{Classification model selection via bilevel programming}.
	{Optimization Methods and Software.} \textbf{23}, 475-489 (2008)
	
	\bibitem{kunapuli2008bilevel}Kunapuli, G.,  Bennett, K.P., Hu, J., Pang, J-S.: Bilevel model selection for support vector machines. In: CRM proceedings and lecture notes. \textbf{45}, 129-158 (2008)
	
	
	
	
	\bibitem{LLNashBilevel}
	Lampariello, L., Sagratella, S.:
	{ Numerically tractable optimistic bilevel problems}.
	Computational Optimization  and Applications. \textbf{76}, 277-303 (2020)
	
	\bibitem{LinXuYe} Lin, G., Xu, M., Ye, J.J.: On solving simple bilevel programs with a nonconvex lower level program. Mathematical Programming. \textbf{144}, 277-305 (2014)
	
	
	
	\bibitem{Liu}
	Liu, R., Mu, P., Yuan, X., Zeng, S., Zhang, J.:
	{A generic first-order algorithmic framework for bi-level programming beyond lower-level singleton.}
	In: International Conference on Machine Learning. \textbf{119}, 6305-6315 (2020)
	
	\bibitem{Liu2}
	Liu, R., Mu, P., Yuan, X., Zeng, S., Zhang, J.:
	{A generic descent aggregation framework for gradient-based bi-level optimization.}
IEEE Transactions on Pattern Analysis and Machine Intelligence. (2022)

	
	
	\bibitem{MPEC1} Luo, Z-Q.,  Pang, J-S.,  Ralph, D.: Mathematical Programs with Equilibrium Constraints.
	Cambridge University Press, Cambridge (1996)
	
	
	
	
	
	\bibitem{Mirrlees} Mirrlees, J.A.: The theory of moral hazard and unobservable behaviour: Part I. The Review of Economic Studies. \textbf{66}, 3-21 (1999)
	
	
	\bibitem{Mooreth}  Moore, G.: Bilevel programming algorithms for machine learning model selection. Ph.D Thesis. (2010)
	\bibitem{Moore} Moore, G., Bergeron, C.,   Bennett, K.P.: Model selection for primal SVM.
	Machine Learning. \textbf{85}, 175-208 (2011)
	
	
	
%
	
	
	
	\bibitem{nie2017bilevel}
	Nie, J., Wang, L., Ye, J.J.:
	{ Bilevel polynomial programs and semidefinite relaxation methods}.
	{ SIAM Journal on Optimization}. \textbf{27}, 1728-1757 (2017)
	
	\bibitem{nie2021bilevel}
	Nie, J., Wang, L., Ye, J.J., Zhong, S.:
	{ A Lagrange Multiplier Expression Method for Bilevel Polynomial Optimization,}
	{SIAM Journal on Optimization}. \textbf{31}(3), 2368-2395 (2021).
	
		\bibitem{HO2020} Okuno,  T.,  Kawana, A.: Bilevel optimization of regularization hyperparameters in machine learning.  In: Bilevel Optimization: Advances and Next Challenges, Ch.~6.
	Springer Optimization and its Applications, vol. 161 (2020).
	
	\bibitem{Outrata1990}  Outrata, J. V.: On the numerical solution of a class of Stackelberg problems, ZOR - Methods and Models of Operations Research. \textbf{34}, 255–277 (1990)
	
	\bibitem{MPEC2} Outrata, J., Kocvara, M., Zowe, J.: Nonsmooth Approach to Optimization Problems with
	Equilibrium Constraints: Theory, Applications and Numerical Results. Kluwer Academic Publishers, Boston (1998)
	

	\bibitem{pang2017} Pang, J.S., Razaviyayn, M. and Alvarado, A.:  Computing B-stationary points of nonsmooth DC programs. Mathematics of Operations Research. \textbf{42(1)}, 95-118 (2017)
	
	
	\bibitem{rockafellar}  Rockafellar, R.T.: Convex Anlysis. Princeton University Press, Princeton (1970)
	\bibitem{rockafellar1974conjugate} Rockafellar, R.T.:  Conjugate duality and optimization. CBMS-NSF Regional Conference Series in Applied Mathematics. \textbf{16}, 1-74 (1974)
	
	\bibitem{Shimizu}
	Shimizu, K., Ishizuka, Y., Bard, J.:
	{Nondifferentiable and Two-level Mathematical Programming.}
	{Kluwer Academic Publishers, Dordrecht} (1997)
	
	\bibitem{Stackelberg} Stackelberg, H.: Market Structure and Equilibrium. Springer Science \& Business Media, Berlin (2010)
	\bibitem{cvxbook} Stephen, B., Vandenberghe, L.: Convex Optimization. Cambridge University Press, Cambridge (2004)
	
	\bibitem{le2014dc} Thi, H.A.L., Dinh, D.T.: Advanced Computational Methods for Knowledge Engineering, DC programming and DCA for general DC programs. pp. 15-35. Springer, Cham, Switzerland (2014)
	
	\bibitem{ThiDinh} Thi, H.A.L., Dinh, T.P.: DC programming and DCA: thirty years of developments. Math. Program. \textbf{169}, 5-68 (2018)
	
	\bibitem{SDPT3}  Toh, K.C.,  Todd,  M.J.,  Tutuncu, R.H.: SDPT3 — a Matlab software package for semidefinite programming, Optimization Methods and Software. \textbf{11}, 545–581 (1999)
	
	\bibitem{SDPT32003}  Tutuncu, R.H.,  Toh, K.C., Todd,  M.J.: Solving semidefinite-quadratic-linear programs using SDPT3. Mathematical Programming, Series B. \textbf{95}, 189–217 (2003)
	
	
	
%
	\bibitem{Yebook} {Ye, J.J.}:
	{ Constraint qualifications and optimality conditions in bilevel optimization}. In: Bilevel Optimization: Advances and Next Challenges, Ch.~8.
	Springer Optimization and its Applications, vol. 161 (2020).
	
	
	\bibitem{ye1995}  Ye J.J., Zhu, D. L.: Optimality conditions for bilevel programming problems. Optimization.
	\textbf{33}, 9–27 (1995)
	
	
	
	
	
	
	
	
	
	
	
	
	
\end{thebibliography}
\end{document}